\documentclass[12pt]{amsart}

\usepackage{amsmath} 
\usepackage{amssymb} 
\usepackage{graphicx} 
\usepackage{ascmac} 

\newtheorem{theorem}{Theorem}[section]
\newtheorem{lemma}[theorem]{Lemma}
\newtheorem{proposition}[theorem]{Proposition}
\newtheorem{corollary}[theorem]{Corollary}

\theoremstyle{definition}
\newtheorem{definition}[theorem]{Definition}

\theoremstyle{remark}
\newtheorem{remark}[theorem]{Remark}

\numberwithin{equation}{section}

\newcommand{\R}{{\mathbb R}}
\newcommand{\N}{{\mathbb N}}
\newcommand{\Z}{{\mathbb Z}}
\newcommand{\ep}{\varepsilon}
\newcommand{\acknowledgments}{Acknowledgment}

\begin{document}

\title[Variational problem associated with the minimal wave speed]
{A variational problem associated with the minimal speed of traveling waves 
for spatially periodic KPP type equations}

\author{Ryunosuke Mori}
\address{Graduate School of Mathematical Sciences, 
University of Tokyo, Tokyo 153-8914, Japan }
\email{moriryu@ms.u-tokyo.ac.jp}

\author{Dongyuan Xiao}
\address{Graduate School of Mathematical Sciences, 
University of Tokyo, Tokyo 153-8914, Japan }
\email{xdyong@ms.u-tokyo.ac.jp}



\begin{abstract}
We consider a variational problem associated with the minimal speed of
pulsating traveling waves of the equation $u_t=u_{xx}+b(x)(1-u)u$, $x\in\R,\ t>0$,
where the coefficient $b(x)$ is nonnegative and periodic in $x\in\R$ with a period $L>0$.
It is known that there exists a quantity $c^*(b)>0$ such that a pulsating traveling wave with the average speed $c>0$ exists if and only if $c\geq c^*(b)$. The quantity $c^*(b)$ is the so-called minimal speed of
pulsating traveling waves.

In this paper, we study the problem of maximizing $c^*(b)$ by varying the coefficient $b(x)$
under some constraints. We prove the existence of the maximizer under a certain assumption of
the constraint and derive the Euler--Lagrange equation which the maximizer satisfies
under $L^2$ constraint $\int_0^L b(x)^2dx=\beta$.
The limit problems of the solution of this Euler--Lagrange equation as $L\rightarrow0$ and
as $\beta\rightarrow0$ are also considered. Moreover, we also consider the variational problem in a certain class of step functions under $L^p$ constraint $\int_0^L b(x)^pdx=\beta$ when $L$ or $\beta$
tends to infinity.
\end{abstract}

\maketitle

\tableofcontents

\label{intro}

We consider the following variational problem:
\begin{equation}\label{eq:maximize}
\begin{split}
\underset{b\in A_f}{\rm Maximize}\ c^*(b),
\end{split}
\end{equation}
where $c^*(b)$ is the minimal speed of pulsating traveling waves for the equation
\[
u_t=u_{xx}+b(x)(1-u)u\ \ (x\in\R,\ t>0)
\]
and
\[
A_f=A_{f,L,\beta}:=\Big\{b\in L^1_{\rm loc}(\R)\,\Big|\,
b(\cdot+L)=b(\cdot),\ b\geq0,\ \frac{1}{L}\int_0^L f(b)dx=\beta\Big\}.
\]

The Reaction-diffusion equation
\begin{equation}\label{eq}
u_t=u_{xx}+g(x,u)
\end{equation}
appears in vast fields of natural sciences including combustion physics, chemical kinetics, biology and so on.

Early in 1937, Fisher \cite{F} and Kolmogorov, Petrovsky and Piskunov \cite{KPP} firstly introduced this model in study the propagation of dominant gene in homogeneous environment. In this paper, we focus on the ecological modeling, which has been first introduced by Skellam \cite{S} in 1951. From the ecological point of view, pulsating traveling waves have been widely used to describe the spatial spread of the territory of a certain species.
In particular, it describes the invasion of the alien species in a given habitat, for example: plants, algae, epidemic agents and so on.
In 1986, Shigesada, Kawasaki and Teramoto \cite{SKT} extended the model (SKT model) naturally for ecological invasions in spatially periodic environments.
More precisely, they considered the heterogeneous environment consisting of two kinds of patches, favorable and less favorable, which are alternately arranged in an infinitely long habitat. In their framework, the reaction term $g(x,u)$ is considered as $b(x)u(1-u)$, where $u$ represents a population density of species, while $b(x)$ respectively correspond to the intrinsic growth rate.
However, from the point of  mathematics, their analysis was partly unrigorous as it relied on formal asymptotic behaviors of the traveling waves near the leading edge.

The first rigorous mathematical result in multi-dimensional homogeneous environment has been given by Aronson and Weinberger \cite{AW}. For the spatially periodic equations, the spreading properties of the solutions has been studied in the work of G$\ddot{\rm a}$rtner and Freidlin \cite{GF}. In particular, if $g$ is positive between $0$ and $1$, $g(x,1)=0$ for all $x$ and $g(x,s)\le g_u(x,0)s$ for all $(x,s)$, there exists a spreading speed $w^*\ge 0$, such that, if the given initial data is non-trivial, nonnegative and compactly supported. For the observers who move slower than $w^*$ will see the population density approaches to a positive state, however, for those who move faster than $w^*$ will see the population density approaches to $0$.

The work of \cite{SKT} has been extended by Berestycki, Hamel and Roques \cite{BHR2} by dealing with a more general equation $u_t=\triangledown\cdot((A(x)\triangledown u))+g(x,u)$ in $\mathbb{R}^n$ with spatially periodic inhomogeneity.
They proved that, under certain assumptions on the coefficients, there exists a constant $c^*>0$, the so-called minimal speed, such that the equation admits pulsating traveling waves with speed $c$ if and only if $c\ge c^*$.
Moreover,
they showed that the minimal speed $c^*$ can be characterized as follows:
\begin{equation}\label{eq:characterization of c^*}
c^*=\min\{c>0\ |\ \exists\lambda>0 \ such\ that\ k(\lambda,b)=-\lambda c\}=\min\limits_{\lambda>0}\frac{-k(\lambda,b)}{\lambda},
\end{equation}
where $b(x)=g_u(x,0)$ and the quantity $k(\lambda,b)$ is the principal eigenvalue $k(\lambda,b)$ of an operator $-L_{\lambda,b}$ defined as follows:
\begin{equation}\label{eq:elliptic operator}
-L_{\lambda,b}\psi=-\frac{d^2}{dx^2}\psi-2\lambda\frac{d}{dx}\psi-(b(x)+\lambda^2)\psi\ \ (x\in\mathbb{R}/L\mathbb{Z}).
\end{equation}

However, when we consider an invading species, it is not only interesting to know whether the species can survive, but also how it survives.
In fact, for a large class of KPP nonlinearities, the minimal speed $c^*$ coincides with $w^*$. Weinberger \cite{W} has studied this property in general multi-dimensional case within the framework of discrete dynamical systems.

In the SKT model \cite{SKT}, the intrinsic growth rate $b(x)$ has been considered as a variant depending on the environmental parameter at the location $x$.
As well-known, the predator's growth rate relates to the density of preys, the levels of available nitrogen and
phosphorus are the most important factors in limiting water plant growth and so on.
In fact, the growth rate of some species are almost nonlinearly depending on the environmental parameter, such as the algae.
A. Dauta et.al \cite{experiment} performed some experiments to study the effect of temperature on the growth rate of the freshwater algae.
They measured the growth rates under different temperature condition($10$, $15$, $20$, $25$, $30^{\circ}C$). From the data of the experiment, it is not
difficult to find that the relationship between growth rate and temperature is nonlinear.

Liang, Lin and Matano \cite{LLM} extended the works of \cite{SKT} by dealing with the equation
\[
u_t=u_{xx}+\bar{b}(x)u(1-u),
\]
where $\bar{b}(x)$ is a nonnegative $L\mbox{-periodic}$ measure on $\mathbb{R}$ with $\bar{b}\not\equiv 0$.
They proved that the existence of the minimal speed still and also proved that the minimal speed coincides with  the spreading speed in this case.
Furthermore, they considered the variational problem of maximizing the minimal speed $c^*(\bar{b})$ under the constraint of the form $\int_0^L\bar{b}=\alpha L$,
where $\alpha>0$ is an given arbitrary constant.
The constraint they dealt with roughly means the average value of the environmental parameter is a given constant
when the growth rate is linearly depending on the environmental parameter.
They found out that the periodically arrayed Dirac's delta functions gives the fastest spreading speed.
Before long, Liang and Matano \cite{LM} extended the above-mentioned work to the two-dimensional stratified media.

Nadin \cite{N} studied the same variational problem by dealing with the effect of the Schwarz rearrangement on the minimal speed.
He proved that $c^*(b^*)\ge c^*(b)$, where $b^*$ is the Schwarz rearrangement of $b$. The definition of Schwartz rearrangement will be given in the section $3$.
Among other things, he found out a new characterization of the principal eigenvalue $k(\lambda,b)$ of the nonsymmetric operator \eqref{eq:elliptic operator} with the periodic boundary condition as follows:
\begin{equation}\label{Nf}
k(\lambda,b)=\min\limits_{\phi\in E_L}\Big\{\int_0^L\phi'^2-\int_0^Lb\phi^2-\frac{\lambda^2L^2}{\int_0^L\phi^{-2}}\Big\},
\end{equation}
where $E_L=\{\phi\ |\ \phi\in H^1_{\rm per},\int_0^L\phi^2=1\}$.
Nadin's formula was extended to the equations when $b(x)$ is a nonnegative measure by Liang and Matano \cite{LM}.
We can conclude from both \cite{LLM} and \cite{N} that the most concentrated $b(x)$ maximizes the minimal speed under the $L_1$ norm constraints.

In this paper, we study the case that the intrinsic growth rate $b(x)$ nonlinearly
depends on an environmental parameter.
The rest of the paper is organized as follows. In Section~2, we state our main results.
In Section~3 of the present paper, we compute the maximizer under the $L^{\infty}$ constraint $A_1$ which is defined as
\[
A_1:=\Big\{b\in L^1_{\rm loc}(\R)\,\Big|\,
b(\cdot+L)=b(\cdot),\ 0\leq b\leq h,\ \frac{1}{L}\int_0^L b dx=\beta\Big\}.
\]
We show that periodically arrayed step functions attains the maximum under this constraint.
In Section~4, we deal with more general constrain condition which is defined as
\[
A_f:=\Big\{b\in L^1_{\rm loc}(\R)\,\Big|\,
b(\cdot+L)=b(\cdot),\ b\geq0,\ \frac{1}{L}\int_0^L f(b)dx=\beta\Big\}.
\]
We prove that the maximizer exists when $f$ is an increasing function and
$\underset{u\rightarrow\infty}{\lim}\,f(b)/b=\infty$.
In Section~5, we derive the Euler-Lagrange equation which the maximizer satisfies when $f(b)=b^2$.
In Section~6, we study the local maximality of the constant $b$ by computing the second variation.
We note that the constant $b_0=\alpha$ is the minimizer of the minimal speed under $L^1$ constraint $\frac{1}{L}\int_0^L b dx=\alpha$.
In Section~7, we study the asymptotic analysis of the maximizers when the period $L$ or the mass $\beta$ is very small and we obtain that the maximizer converges to a constant as $L\rightarrow0$ or as $\beta\rightarrow0$.
We also study the asymptotic analysis of maximizers in the set which consists of periodically arrayed step functions $b$ under the constraint $\frac{1}{L}\int_0^L b^p dx=\beta$ when $L$ or $\beta$ is very large. We also obtain that the maximizer is concentrated as $L\rightarrow\infty$ or as $\beta\rightarrow\infty$.

\section{Main theorems}

\subsection{Existence of maximizer under general constraints}
For $L>0$, $\beta>0$ and a function $f$, we define the nonlinear constraint $A_f$ as follow:
\[
A_f=A_{f,L,\beta}:=\Big\{b\in L^1_{\rm loc}(\R)\,\Big|\,
b(\cdot+L)=b(\cdot),\ b\geq0,\ \frac{1}{L}\int_0^L f(b)dx=\beta\Big\},
\]
In order to state the existence result, we need to impose an assumption
We assume that the function $f\in C([0,\infty))$
is an increasing function such that
\begin{equation}\label{cd:superlinear}
\lim_{u\rightarrow\infty}\frac{f(b)}{b}=\infty.
\end{equation}
It is not difficult to check that $f(b)=b^p$ satisfies this assumption when $p>1$.
Based on this assumption, we get the following theorem about the existence of the maximizer.

\begin{theorem}
[Existence of the maximizer]
\label{th:existence}
The the functional $b\mapsto c^*(b)$ attains its maximum in $A_f$ provided that
$f$ is an increasing continuous function satisfying \eqref{cd:superlinear}.
\end{theorem}
\subsection{Maximizer under $L^1$ and $L^{\infty}$ constraints}
\begin{theorem}\label{th:existence L^1 L^{infty}}
Let $\alpha$ be any given constants such that $0<\alpha< h$
and set
\[
\begin{split}
A_1&:=\Big\{
b\in L^1_{L\mbox{\tiny-per}}\,\Big|\,
0\leq b\leq h,\ \int_0^L b dx=\alpha L
\Big\},\\
b_{1}(x)
&:=\sum_{k\in{\mathbb Z}}h\chi_{I_{k}}(x),\
I_{k}
:=\Big(k+\frac{1}{2}\Big)L
+\Big[-\frac{\alpha L}{2h},\frac{\alpha L}{2h}\Big].
\end{split}
\]
Then
\[
c^{*}(b_{1})=\max_{b\in A_{1}}c^{*}(b).
\]
Moreover, the maximizer of
$c^{*}(b)$
in
$A_{1}$
is unique up to translation in
$x$.
\end{theorem}
\begin{remark}
From Theorem \ref{th:existence L^1 L^{infty}}, we can see that the most concentrated function in $A_1$ attains the maximum of
$c^*$. This property corresponds to the results of \cite{LLM} and \cite{N}.
\end{remark}
\subsection{$L^2$ constraints and the Euler--Lagrange equation}
In general, how the maximizer can not be computed explicitly. Therefore, we need to derive the Euler--Lagrange equation to study to the properties of the maximizer.
\begin{theorem}
[Euler--Lagrange equation]
\label{th:Euler--Lagrange}
Let $b(x)$ be a maximizer of $c^*$ in $A_2$. Then $b$ satisfies
the Euler--Lagrange equation:
\begin{equation}\label{eq:EL equation}
\left\{
\begin{split}
&b''+4k b+3 b^2= C\ \ (x\in\R/L\Z),\\
&C=3\beta+\frac{4 k}{L}\int_0^Lb dx,\ k=k(\lambda(b),b).
\end{split}
\right.
\end{equation}
\end{theorem}
\begin{remark}From the general theory of
second order ordinary differential equations,
any nontrivial solution of
the above equation has a minimum period
$L/N$
for some
$N\in{\mathbb N}$ and
is symmetric with respect to its extreme point.
Hence, as a consequence of a corollary of Theorem 1.2 of \cite{N},
the equation \eqref{eq:EL equation} implies that
any non-trivial maximizer $b\in A_2$ satisfies
\[
b'(x)>0\ \ (x\in(a-L/2,a)),\ \
b(\cdot)=b(2a-\cdot)\ \ {\rm for\ some}\ \ a\in[0,L).
\]
\end{remark}
\subsection{Constant coefficients as local maximizers}
A direct result from Theorem \label{th:Euler--Lagrange} that the Euler--Lagrange equation obtains constant solution motivates us to check whether it is the local maximizer or not.
\begin{theorem}
[Local maximality of the constant intrinsic growth rate under $L^{\infty}$-perturbation]
\label{th:local maximality L^{infty}}
Let $f\in C^2((0,\infty))$ be a strictly increasing function with $f''(f^{-1}(\beta))>0$.
If
\[
D(\beta):=2f''(f^{-1}(\beta))(\pi^2/L^2+f^{-1}(\beta))-f'(f^{-1}(\beta))>0,
\]
then the constant function $b_0:=f^{-1}(\beta)$
is a local maximizer of $c^*$ in $A_{f,L,\beta}$. Namely,
there is an $\eta>0$ such that
\[
c^*(b_0)>c^*(b)\ \ {\rm for}\ \ b\in A_{f,L,\beta}\ \
{\rm with}\ \ 0<\|b_0-b\|_{L^{\infty}}\leq \eta.
\]
If $D(\beta)<0$,
then the constant function $b_0:=f^{-1}(\beta)$
is neither a local maximizer nor a local minimizer of $c^*$ in $A_{f,L,\beta}$.
\end{theorem}

As a corollary, we can compute the condition whether the constant is a local maximizer when $f(b)=b^P$.
\begin{corollary}[$L^p$-constraints]
\label{cr:L^p}

The constant function $b_0:=\beta^{1/p}$
is a local maximizer of $c^*$ in $A_{p,L,\beta}$, more precisely,
there exists an $\eta>0$ such that
\[
c^*(b_0)>c^*(b)\ \ {\rm for\ any}\ \ b\in A_{p,L,\beta}\ \
{\rm with}\ \ 0<\|b_0-b\|_{L^{\infty}}\leq \eta.
\]
if $p\geq 3/2$ or
\[
1<p<3/2\ \ \mbox{and}\ \ \beta^{1/p}<\frac{2(p-1)\pi^2}{(3p-2)L^2}.
\]
Otherwise, the constant function $b_0:=\beta^{1/p}$
is neither a local maximizer nor a local minimizer of $c^*$ in $A_{p,L,\beta}$
if

\[
1<p<3/2\ \ \mbox{and}\ \ \beta^{1/p}>\frac{2(p-1)\pi^2}{(3-2p)L^2},
\]
\end{corollary}

\begin{theorem}
[Local maximality of the constant intrinsic growth rate under $L^2$-perturbation]
\label{th:local minimality L^2}
The constant function $b_0:=\sqrt{\beta}$ is a local maximizer of $c^*$ in $A_2$,
more precisely,
there exists an $\eta>0$ such that
\[
c^*(b_0)>c^*(b)\ \ {\rm for}\ \ b\in A_2\ \
{\rm with\ any}\ \ 0<\|b_0-b\|_{L^{2}([0,L])}<\eta.
\]
\end{theorem}
\subsection{Various asymptotic analyses}
We are also interested in the asymptotic behavior of the maximizer as $L->0$ or $\infty$ and $\beta->0$ or $\infty$.
\begin{theorem}[The case of $L\rightarrow0$]
\label{th:L->0}
Let $\{b_L\}_{L>0}$ be a sequence such that for each $L>0$,
\[
b_L\in A_{2,L,\beta}:=\Big\{b\in L^1_{\rm loc}(\R)\,\Big|\,
b(\cdot+L)=b(\cdot),\ b\geq0,\ \frac{1}{L}\int_0^L b^2dx=\beta
\Big\}
\]
satisfies $c^*(b_L)=\underset{b\in A_{2,L,\beta}}{\max}\,c^*(b)$. Then
for any $k\in\N$,
\[
\big\|b_L-\sqrt{\beta}\big\|_{C^2}=O(L^k)\ \ {\rm as}\ \ L\rightarrow0.
\]
\end{theorem}
\begin{theorem}[The case of $\beta\rightarrow0$]
\label{th:beta->0}
Let $\{b_\beta\}_{\beta>0}$ be a sequence such that for each $\beta>0$,
\[
b_{\beta}\in A_{2,L,\beta}:=\Big\{b\in L^1_{\rm loc}(\R)\,\Big|\,
b(\cdot+L)=b(\cdot),\ b\geq0,\ \int_0^L b^2dx=\beta
\Big\}
\]
satisfies $c^*(b_{\beta})=\underset{b\in A_{2,L,\beta}}{\max}\,c^*(b)$. Then
for any $k\in\N$,
\[
\big\|b_{\beta}-\sqrt{\beta}\big\|_{C^2}=O(\beta^k)\ \ {\rm as}\ \ \beta\rightarrow0.
\]
\end{theorem}

Different from the rapidly oscillating environments, the minimal speed is difficult to characterize for the slowly oscillating environments. However, we can still deal with the asymptotic analyses by assuming the maximizer is periodically arrayed step functions.
\begin{theorem}[The case of $L\rightarrow\infty$]
\label{th:L->infty}
Define the step function $b_{\theta,L}(x)$ as 
\[
b_{\theta,L}(x):=b_{\theta}(x/L),\ \
b_{\theta}(x):=\beta^{1/p}\sum_{k\in\Z}\theta^{-1/p}\chi_{I_k}(x)\ \ (I_k:=k+[0,\theta])
\]
and choose $\theta(L)\in(0,1]$ such that
$c^*(b_{\theta(L),L})=\underset{\theta\in(0,1]}{\max}\,c^*(b_{\theta,L})$. Then
\[
\theta(L)\rightarrow0,\ \
c^*(b_{\theta(L),L})\rightarrow\infty\ \ {\rm as}\ \ L\rightarrow\infty.
\]
\end{theorem}
\begin{theorem}[The case of $\beta\rightarrow\infty$]
\label{th:beta->infty}
Define the step function $b_{\theta}(x)$ as 
\[
b_{\theta}(x):=\sum_{k\in\Z}\theta^{-1/p}\chi_{I_k}(x/L)\ \ (I_k:=k+[0,\theta])
\]
and choose $\theta(\beta)\in(0,1]$ such that
$c^*(\beta^{1/p}b_{\theta(\beta)})=\underset{\theta\in(0,1]}{\max}\,c^*(\beta^{1/p}b_{\theta})$.
Then
\[
\theta(\beta)\rightarrow0,\ \
c^*(\beta^{1/p}b_{\theta(\beta)})/\beta^{\frac{1}{2p}}
\rightarrow\infty\ \ {\rm as}\ \ \beta\rightarrow\infty.
\]
\end{theorem}
\begin{remark}
We note that for each $L>0$ and each $\beta>0$,
$\underset{\theta\rightarrow0}{\lim}\,c^*(\beta^{1/p}b_{\theta,L})=0$.
\end{remark}
\begin{remark}
As a consequence of Corollary \ref{cr:L^p}, for each $p\in(1,\infty)$, the constant function $b_0:=\beta^{1/p}$
is a local maximizer of $c^*$ in $A_{p,L,\beta}$ when $\beta>0$ or $L>0$ is very large.
From Theorems \ref{th:L->infty} and \ref{th:beta->infty},
for sufficiently large $L>0$ or for sufficiently large $\beta>0$,
the constant $b_0$ is not a global maximizer and hence the maximizer is a non-trivial function.
Especially, we can conclude that the Eular--Lagrange equation (2.1) has non-trivial solution when $L>0$ or $\beta>0$
is large enough.
\end{remark}
\section{Proof of the result under $L^1$ and $L^{\infty}$ constraints}
In this section, we mainly consider the maximizer of the minimal pulsating traveling wave speed
under $L^1$ and $L^{\infty}$ constraints. We begin by introducing the Schwarz rearrangement. Nadin's formula and the properties of Schwarz rearrangement
play the most important roles in the proof.

\begin{definition}
[The Schwarz periodic rearrangement]
\label{df:Schwarz}
For an $L$-periodic measurable function $\phi$,
there exists a unique $L$-periodic function
$\phi^{*}$
such that
\begin{itemize}
\item[(i)]
$\phi^{*}$
is symmetric with respect to
$L/2$,
\item[(ii)]
$\phi^{*}$
is nondecreasing in
$[0,L/2]$,
\item[(iii)]
for all
$a\in{\mathbb R}$,
\[
{\rm meas}\{x\in[0,L]\mid \phi(x)>a\}
={\rm meas}\{x\in[0,L]\mid \phi^{*}(x)>a\}
\]
\end{itemize}
The function
$\phi^{*}$
is called the Schwarz periodic rearrangement of the function
$\phi$.
\end{definition}

We recall that the minimal speed of traveling waves is characterized by
\[
c^{*}(b)
=\min_{\lambda>0}\frac{-k(\lambda,b)}{\lambda},
\]
where
\[
-L_{\lambda,b}\psi
:=-\psi''-2\lambda\psi'-(b(x)+\lambda^{2})\psi
=k(\lambda,b)\psi
\quad(x\in{\mathbb R/L{\mathbb Z}}).
\]
As we mentioned in the section $1$, the following formula (Nadin's formula) holds:
\[
k(\lambda,b)=\min_{\phi\in E_{L}}H(\phi,\lambda,b),
\]
\[
H(\phi,\lambda,b)=
H_{L}(\phi,\lambda,b):=
\int_{0}^{L}(\phi^{\prime2}-b\phi^{2})dx
-\frac{\lambda^{2}L^{2}}{\int_{0}^{L}\phi^{-2}dx},
\]
\[
E_{L}
:=\Big\{\phi\in H^{1}_{L\mbox{\tiny-per}}\,\Big|\,\int_{0}^{L}\phi^{2}=1\Big\}.
\]

\begin{remark}
[Consequences of Nadin's formula]
The following propositions are easily proved by using Nadin's formula.
\begin{proposition}
[H. Berestycki, F. Hamel and L. Roques \cite{BHR1},
Nadin \cite{N}]
\label{pr:S-rearrangement effect}
Fix an $L$-periodic function $b$.
Under the above notations, we have
\[
k(\lambda,b)\geq k(\lambda,b^*)\quad\quad(\lambda\in{\mathbb R}),
\]
and hence
\[
c^*(b)\leq c^*(b^*).
\]
\label{equi-measurable}
\end{proposition}
\begin{proposition}
[H. Berestycki, F. Hamel and L. Roques \cite{BHR2}]
\label{pr:concavity}
For each
$b\in A_{f}$,
the function
\[
\lambda\mapsto k(\lambda,b)
\]
is strictly concave and
$k(0,b)<0$,
especially,
there exists a unique
$\lambda(b)>0$
satisfying
\[
\frac{-k(\lambda(b),b)}{\lambda(b)}
=\min_{\lambda>0}\frac{-k(\lambda,b)}{\lambda}.
\]
\end{proposition}
\begin{proposition}
[G. Nadin \cite{N}]
\label{pr:monotonicity}
The functions
\[
L\mapsto\sup_{b\in A_{L,f}}c^*(b)
\]
are monotonically increasing.
\end{proposition}
\begin{proposition}
[H. Berestycki, F. Hamel and L. Roques \cite{BHR2}]
\label{pr:minimizer in L^1}
For any
$b\in A_{f}$,
\[
2\Big(\frac{1}{L}\int_0^L b dx\Big)^{1/2}\leq c^{*}(b),
\]
and the equality holds if and only if
$b$ is constant.
\end{proposition}
\end{remark}

\begin{proof}[Proof of Theorem \ref{th:existence L^1 L^{infty}}]
Define
\[
E_{L}^{*}:=\{\phi\in E_{L}\mid \phi^{*}=\phi\},\
A_{1}^{*}:=\{b\in A_{1}\mid b=b^{*}\}.
\]
Here we recall that
$u^{*}$
is the Schwarz periodic rearrangement of
an $L$-periodic measurable function $u$.
First we prove that
\[
c^*(b_1)=\max_{b\in A_1}\,c^*(b),
\]
where
$b_{1}:=\underset{k\in{\mathbb Z}}{\sum}h\chi_{I_{k}}$,
$I_{k}=(k+1/2)L+[-\alpha L/(2h),\alpha L/(2h)]$.
Applying the Polya inequality
\[
\int_{0}^{L}\phi^{*\prime2}dx
\leq\int_{0}^{L}\phi^{\prime2}dx
\]
and the Hardy--Littlewood inequality, we get that
\[
\int_{0}^{L}b\phi^{2}dx
\leq\int_{0}^{L}b^{*}\phi^{*2}\ \
(b\geq0, \phi\geq0),
\]
\begin{equation}\label{eq:minmin}
\begin{split}
\min_{b\in A_{1}}&\,k(\lambda,b)
=\min_{b\in A_{1}}\min_{\phi\in E_{L}}H_{L}(\phi,\lambda,b)
=\min_{\phi\in E_{L}^{*}}\min_{b\in A_{1}^{*}}
H_{L}(\phi,\lambda,b)\\
=&\min_{\phi\in E_{L}^{*}}
\Big\{
\int_{0}^{L}\phi^{\prime2}dx
-\max_{b\in A_{1}^{*}}
\Big\{\int_{0}^{L}b\phi^{2}dx\Big\}
-\frac{\lambda^{2}L^{2}}{\int_{0}^{L}\phi^{-2}dx}
\Big\}.
\end{split}
\end{equation}
For the reason that
$\phi\in A_{1}^{*}$
is symmetric with respect to
$x=L/2$
and nondecreasing in
$(0,L/2)$,
\begin{equation}
\label{eq:maximality b_{1}}
\int_{0}^{L}b_{1}\phi^{2}dx
=\max_{b\in A_{1}^{*}}\int_{0}^{L}b\phi^{2}dx.
\end{equation}
Indeed, for $b\in A_1$, since $\int_0^Lbdx=\int_0^L b_1 dx$ and $0\leq b(x)\leq h$,
\[
\begin{split}
&b_1(x)-b(x)\geq0\ \ {\rm for}\ \
(1-\alpha/h)\frac{L}{2}\leq x\leq(1+\alpha/h)\frac{L}{2},\\
&\int_{(1-\alpha/h)\frac{L}{2}}^{(1+\alpha/h)\frac{L}{2}}(b_1-b)dx
=\int_0^{(1-\alpha/h)\frac{L}{2}}bdx
+\int_{(1+\alpha/h)\frac{L}{2}}^L bdx.
\end{split}
\]
Thus, if we put $l_*:=(1-\alpha/h)\frac{L}{2}$, $l^*:=(1+\alpha/h)\frac{L}{2}$, then
\[
\begin{split}
\int_0^L b\phi^2dx
&\leq\int_{l_*}^{l^*}b\phi^2dx
+\phi(l_*)^2\Big(\int_0^{l_*}bdx
+\int_{l^*}^L bdx\Big)\\
&
=\int_{l_*}^{l^*}b\phi^2dx
+\phi(l_*)^2\int_{l_*}^{l^*}(b_1-b)dx\\
&\leq\int_{l_*}^{l^*}b\phi^2dx
+\int_{l_*}^{l^*}(b_1-b)\phi^2dx=\int_0^L b_1\phi^2dx.
\end{split}
\]
Hence \eqref{eq:maximality b_{1}} holds.
Therefore for all
$\lambda>0$
and
$b\in A_{1}$,
\[
k(\lambda,b)\geq k(\lambda,b_{1}).\ \ \therefore
-k(\lambda,b)/\lambda\leq-k(\lambda,b_{1})/\lambda.
\]
\[
\therefore
c^{*}(b)
=\min_{\lambda>0}\,(-k(\lambda,b)/\lambda)
\leq\min_{\lambda>0}\,(-k(\lambda,b_{1})/\lambda)
=c^{*}(b_{1}).
\]
Thus, the maximum of
$c^{*}$
in
$A_{1}$
is attained at the function
$b_{1}$.

Next we prove the uniqueness of the maximizer.
Let
$b_{2}\in A_{1}$
be another maximizer of
$c^{*}$
in
$A_{1}$.
First we claim that
\begin{equation}\label{eq:symmetry of b_{2}}
{\rm there\ exists\ an}\ \
x_{0}\in[0,L) \ \
{\rm such\ that}\ \
b_{2}(\cdot-x_{0})^{*}=b_{2}(\cdot-x_{0}).
\end{equation}
\
The claim \ref{eq:symmetry of b_{2}} is proved by contradiction.
Assume that
\[
\,b_{2}(\cdot-x_{0})^{*}
\not=b_{2}(\cdot-x_{0})
\ \ {\rm for}\ \
x_{0}\in[0,L).
\]
For arbitrarily chosen $\lambda>0$,
let
$\phi_{\lambda}=\phi_{\lambda,b_2}\in E_{L}$
be the function such that
\[
H(\phi_{\lambda},\lambda,b_{2})
=\underset{\phi\in E_{L}}{\min}H(\phi,\lambda,b_{2})
=k(\lambda,b_{2}).
\]
By translating
$b_{2}$
and
$\phi_{\lambda}$
in $x$ simultaneously,
we may assume that
$\phi_{\lambda}$
attains its maximum at
$x=L/2$.

If assume $\phi(L/2)=\underset{x\in[0,L]}{\max}\,\phi(x)$,
then, for the Polya inequality
\[
\int_0^L|\phi^{\prime}|^2dx\geq\int_0^L|\phi^{*\prime}|^2dx,
\]
the equality holds if and only if
$\phi^{*}=\phi$
and, for the Hardy--Littlewood inequality
\[
\int_0^L b\phi^2 dx\leq\int_0^L b^*\phi^{*2}dx,
\]
the equality holds if and only if
$\phi^{*}=\phi$
and
$b^{*}=b$.
Thus, by the assumption
$b_{2}^{*}\not=b_{2}$, we know that
\[
k(\lambda,b_{2})
=H(\phi_{\lambda},\lambda,b_{2})
>H(\phi_{\lambda}^{*},\lambda,b_{2}^{*})
\geq k(\lambda,b_{2}^{*}).
\]
Since
$\lambda>0$
is arbitrary, it is clear that
\[
\begin{split}
c^{*}(b_{2})
&=\min_{\lambda>0}\frac{-k(\lambda,b_{2})}{\lambda}
\leq\frac{-k(\lambda(b_2^*),b_2)}{\lambda(b_2^*)}\\
&<\frac{-k(\lambda(b_2^*),b_2^*)}{\lambda(b_2^*)}
=\min_{\lambda>0}\frac{-k(\lambda,b_{2}^{*})}{\lambda}
=c^{*}(b_{2}^{*}).
\end{split}
\]
This contradicts
$c^{*}(b_{2})
=\underset{b\in A_{1}}{\max}\,c^{*}(b)$.
Thus claim \eqref{eq:symmetry of b_{2}} holds.

By claim \eqref{eq:symmetry of b_{2}}, translating
$b_{2}$
in
$x$,
we may assume that
$b_{2}^{*}=b_{2}$.
Moreover, for the same reason as
in the proof of claim \eqref{eq:symmetry of b_{2}}, it is clear that
$\phi_{\lambda}^{*}=\phi_{\lambda}$.
Thus, if
$b_{2}\not=b_{1}$,
then we get that
$\int_{0}^{L}b_{2}\phi_{\lambda}^{2}dx
<\int_{0}^{L}b_{1}\phi_{\lambda}^{2}dx$
by using similar argument to that in the proof of \eqref{eq:maximality b_{1}}.
Hence
\[
\begin{split}
k(\lambda,b_{2})
&=\int_0^L\phi_{\lambda}^{\prime 2}dx-\int_0^L b_2\phi_{\lambda}^2dx
-\frac{\lambda^2L^2}{\int_0^L\phi_{\lambda}^{-2}dx}\\
&>\int_0^L\phi_{\lambda}^{\prime 2}dx-\int_0^L b_1\phi_{\lambda}^2dx
-\frac{\lambda^2L^2}{\int_0^L\phi_{\lambda}^{-2}dx}
\geq k(\lambda,b_{1})\ \ (\lambda>0).
\end{split}
\]
Therefore
\[
c^*(b_1)=\frac{-k(\lambda(b_1),b_1)}{\lambda(b_1)}>\frac{-k(\lambda(b_1),b_2)}{\lambda(b_1)}
\geq c^*(b_1).
\]
This contradicts
$c^{*}(b_{2})=\underset{b\in A_{1}}{\min}\,c^{*}(b)=c^{*}(b_{1})$.
The uniqueness is proved and the proof is completed.
\end{proof}
\section{Proof of existence of the maximizer under general constraints}
In this section, we prove Theorem \ref{th:existence}, the existence of the maximizer under the constraint
$\frac{1}{L}\int_0^L f(b)dx=\beta$, where $f$ satisfies the assumption \eqref{cd:superlinear}.
Throughout this paper, we shall define the positive functions
$\psi_{\lambda,b}$,
$\widetilde{\psi}_{\lambda,b}$
as the $L$-periodic principal eigenfunction of the operators
\[
\begin{split}
-L_{\lambda,b}=
&-\frac{d^{2}}{d x^{2}}-2\lambda\frac{d}{d x}
-(b(x)+\lambda^{2})I,\\
-L_{\lambda,b}^{*}
=&-\frac{d^{2}}{d x^{2}}+2\lambda\frac{d}{d x}
-(b(x)+\lambda^{2})I,
\end{split}
\]
satisfying
\[
\|\psi_{\lambda,b}\|_{L^2([0,L])}=\langle\widetilde{\psi}_{\lambda,b},\psi_{\lambda,b}\rangle=1
\]
where
$I$
is identity and
\[
\langle\phi_{1},\phi_{2}\rangle
=\int_{0}^{L}\phi_{1}\phi_{2}d x,\
\|\phi\|_{L^p([0,L])}^{p}=\int_{0}^{L}|\phi|^{p}d x.
\]
\begin{remark}[\cite{N}]
\label{rm:phi=sqrt{psiwidetilde{psi}}}
Let
$\phi_{\lambda,b}\in E_{L}$
be the nonnegative function such that
\[
k(\lambda,b)=H(\phi_{\lambda,b},\lambda,b)
=\min_{\phi\in E_{L}}H(\phi,\lambda,b).
\]
Then
we can find that
$\phi_{\lambda,b}
=\sqrt{\psi_{\lambda,b}\widetilde{\psi}_{\lambda,b}}$
from the proof of Nadin's formula in \cite{N}.
\end{remark}
\begin{remark}[\cite{LLM}]
Define $\Lambda:=\{\nu\mid \nu\ {\rm is\ a\ Radon\ measure\ on}\ \R\ {\rm such}$ that $\nu(L+A)=\nu(A)\ {\rm for}\ A\subset\R\ {\rm and}\
\nu([0,L))\geq0.\}$,
we can extend the domain of functionals $k(\lambda,b)$, $\lambda(b)$,
$c^*(b):=-k(\lambda(b),b)/\lambda(b)$
to $\nu\in\Lambda$ (see Proposition 2.15 of Liang, Lin and Matano \cite{LLM})
Otherwise, the following proposition holds:
\begin{proposition}
[Propositions 2.20 and 4.7 of \cite{LLM}]
\label{pr:w^*-continuity}
Let
$\{\nu_n\}\subset\Lambda$
be a sequence converging to some
$\nu\in\Lambda$
in the $\mbox{weak}^*$ sense, that means,
\[
\lim_{n\rightarrow\infty}\int_{\mathbb R}\phi(x) d\nu_n(x)
=\int_{\mathbb R}\phi(x)d\nu(x)\
\mbox{ for any }\phi\in C^{\infty}_{0}(\mathbb R).
\]
Then
\[
\lim_{n\rightarrow\infty}k(\lambda,\nu_n)=k(\lambda, \nu)
\]
locally uniformly in
$\lambda\in{\mathbb R}$,
and
\[
\lim_{n\rightarrow\infty}\lambda(\nu_n)=\lambda(\nu),\
\lim_{n\rightarrow\infty}c(\nu_n)=c(\nu).
\]
\end{proposition}
\end{remark}

Before proving Theorem \ref{th:existence},
we calculate the derivative of the minimal pulsating traveling wave speed $c^*(b)$.
The regularity of $c^*(b)$ comes from the following two lemmas.
The proof of Lemma \ref{lm:smoothness of k} will be given in the appendix.
\begin{lemma}[G$\hat{\rm a}$teaux derivative of $k(\lambda,b)$]
\label{lm:smoothness of k}
For $p\in[1,\infty)$, the maps
\[
\begin{split}
&\R\times L^p_{L\mbox{\rm\tiny-per}}\ni(\lambda,b)\mapsto k(\lambda,b)\in\R,\\
&\R\times L^p_{L\mbox{\rm\tiny-per}}\ni
(\lambda,b)\mapsto \psi_{\lambda,b},\widetilde{\psi}_{\lambda,b}\in W^{2,p}_{L\mbox{\rm\tiny-per}}
\end{split}
\]
are analytic and
\begin{equation}\label{eq:derivative of k}
\begin{split}
\partial_b k(\lambda,b)[v]
&:=\lim_{h\rightarrow0}\frac{k(\lambda,b+h v)-k(\lambda,b)}{h}\\
&=-\langle\psi_{\lambda,b}\widetilde{\psi}_{\lambda,b},v\rangle
=-\int_0^L\psi_{\lambda,b}\widetilde{\psi}_{\lambda,b}v d x.
\end{split}
\end{equation}
\end{lemma}
\begin{lemma}\label{lm:smoothness of lambda}
Let $\lambda(b)$ be the functional defined as
\[
\frac{-k(\lambda(b),b)}{\lambda(b)}=\min_{\lambda>0}\frac{-k(\lambda,b)}{\lambda}.
\]
Then $L^{1,+}_{L\mbox{\rm\tiny-per}}\ni b\mapsto\lambda(b)$ is analytic and
\begin{equation}\label{eq:derivative of lambda}
\partial \lambda(b)[v]
=\frac{\lambda(b)\partial_{b,\lambda}^2k(\lambda(b),b)[v]-\partial_b k(\lambda(b),b)[v]}
{\lambda(b)\partial^2_{\lambda} k(\lambda(b),b)}.
\end{equation}
\end{lemma}
\begin{proof}
The positive quantity
$\lambda(b)$
is uniquely determined by
\[
-k(\lambda(b),b)/\lambda(b)
=\min_{\lambda>0}\,(-k(\lambda,b)/\lambda)
\ (=c^{*}(b)).
\]
Thus
$\lambda(b)>0$
satisfies
\[
0=
\frac{d}{d\lambda}(-k(\lambda,b)/\lambda)_{\lambda=\lambda(b)}
=-\partial_{\lambda}k(\lambda(b),b)/\lambda(b)
+k(\lambda(b),b)/\lambda(b)^{2}.
\]
\begin{equation}\label{eq:implicit of lambda(b)}
\therefore
F(\lambda(b),b):=
-\lambda(b)\partial_{\lambda}k(\lambda(b),b)+k(\lambda(b),b)=0.
\end{equation}
Since Lemma \ref{lm:smoothness of k}, we can conclude that the function
${\mathbb R}\times L^{1}_{L\mbox{\tiny-per}}
\ni(\lambda,b)\mapsto F(\lambda,b)\in{\mathbb R}$
is analytic.
Moreover, from Proposition \ref{pr:concavity}, for
$b\in L^{1,+}_{L\mbox{\tiny-per}}$, we get
\[
\partial_{\lambda}F(\lambda(b),b)
=
\frac{d}{d\lambda}
(-\lambda \partial_{\lambda}k(\lambda,b)
+k(\lambda,b))_{\lambda=\lambda(b)}
=
-\lambda(b)\partial_{\lambda}^{2}k(\lambda(b),b)
>0.
\]
Therefore the implicit function theorem implies that
the function
$L^{1,+}_{L\mbox{\tiny-per}}\ni b\mapsto
\lambda(b)\in(0,\infty)$
is analytic and
\[
\partial\lambda(b)=
\frac{-\partial_b F}{\partial_{\lambda}F}
=\frac{
-\lambda(b)\partial^2_{b,\lambda}k(\lambda(b),b)
+\partial_{b}k(\lambda(b),b)}
{\lambda(b)\partial^2_{\lambda}k(\lambda(b),b)}.
\]
\end{proof}
\begin{proposition}
\label{pr:derivative of c(b)}
The functional
\[
L^{1,+}_{L\mbox{\tiny-per}}(\R)\ni b\mapsto c^*(b):=-k(\lambda(b),b)/\lambda(b)
\]
is analytic and
\begin{equation}\label{eq:derivative of c(b)}
\partial c^*(b)[v]
=\frac{1}{\lambda(b)}\langle\psi_b\widetilde{\psi}_b,v\rangle,
\end{equation}
where $\psi_b=\psi_{\lambda(b),b}$,
$\widetilde{\psi}_b=\widetilde{\psi}_{\lambda(b),b}$.
\end{proposition}
\begin{proof}
The analyticity of the functional
\[
L^{1,+}_{L\mbox{\tiny-per}}(\R)\ni b\mapsto c^*(b)=-k(\lambda(b),b)/\lambda(b)
\]
is a clear consequence of Lemmas \ref{lm:smoothness of k} and \ref{lm:smoothness of lambda}.
The results we computed before \eqref{eq:derivative of k}, \eqref{eq:derivative of lambda} and \eqref{eq:implicit of lambda(b)} implies that
\[
\begin{split}
\partial c^*(b)[v]=&\frac{-1}{\lambda(b)}\partial_b k(\lambda(b),b)[v]
-\frac{\partial\lambda(b)[v]}{\lambda(b)^2}
\big\{\lambda(b)\partial_{\lambda}k(\lambda(b),b)-k(\lambda(b),b)\big\}\\
=&\frac{-1}{\lambda(b)}\partial_b k(\lambda(b),b)[v]
=\frac{1}{\lambda(b)}\langle\psi_b\widetilde{\psi}_b,v\rangle.
\end{split}
\]
\end{proof}
Now we can start to prove Theorem \ref{th:existence}.
\begin{proof}
[Proof of Theorem \ref{th:existence}]
Let $\{b_n\}\subset A_f$ be a maximizing sequence of $c^*(b)$, namely,
$\underset{n\rightarrow\infty}{\lim}\,c^*(b_n)=\underset{b\in A_f}{\sup}\,c^*(b)$.
From Proposition \ref{pr:S-rearrangement effect}, it is clear that $b_n(x)$ is symmetric with respect to
$x=L/2$ and nondecreasing in $x\in[0,L/2]$ for each $n\in\N$.
Since $\int_0^L f(b_n)dx=\beta L<\infty$ and $\underset{b\rightarrow\infty}{\lim}\,f(b)/b=\infty$,
we get that
\[
\sup_{n\in\N}\int_{\{x\in[0,L]\mid b_n(x)\geq N\}} b_n dx
\leq \sup_{n\in\N}\int_0^L b_n dx\, \sup_{b\geq N}\frac{b}{f(b)}
=\beta L\,\sup_{b\geq N}\frac{b}{f(b)}\overset{N\rightarrow\infty}{\longrightarrow}0.
\]
Hence, we can conclude that $\{b_n\}$ is uniformly integrable.

From the periodicity,
the uniform integrability and the monotonicity of $\{b_n\}$ in $x\in[0,L/2]$,
there exists a subsequence $\{b_{n_k}\}$ and an $L$-periodic nonnegative function $b_*$ such that
\[
b_{n_k}\rightarrow b_*\ \ {\rm in}\ \ L^1_{\rm loc},\ \
b_{n_k}(x)\rightarrow b_*(x)\ \ {\rm a.e.}\ \ x\in[0,L]\ \
{\rm as}\ \ k\rightarrow\infty.
\]
By Proposition \ref{pr:w^*-continuity}, we have that
\[
0<\sup_{b\in A_f}\,c^*(b)=\lim_{k\rightarrow\infty}\,c^*(b_{n_k})=c^*(b_*),
\]
and $b_*\not=0$ because of $c^*(0)=0$.

On the other hand, by using Fatou's lemma,
$\underset{k\rightarrow\infty}{\lim}\,f(b_{n_k}(x))=f(b_*(x))$ a.e. on $[0,L]$ implies that
\[
0<\int_0^L f(b_*)dx\leq\liminf_{k\rightarrow\infty}\int_0^L f(b_{n_k})dx=\beta L.
\]
If it is true that $\int_0^L f(b_*)dx=\beta L$, then we can conclude immediately  $b_*\in A_f$ is a maximizer of $c^*(b)$.
Now we prove this argument by contradiction.

Suppose $\int_0^L f(b_*)dx<\beta L$ and define $b_{\ep}(x):=\max\{b_*(x),\ep\}$.
Since $\int_0^L f(b_{0})dx=\int_0^L f(b_*)dx<\beta L$ and $\underset{\ep\rightarrow\infty}{\lim}\,\int_0^L f(b_{\ep})dx=\infty$,
from the continuity of $\ep\mapsto\int_0^L f(b_{\ep})dx$,
we find that there exists an $\ep_*>0$ such that
\[
b_{\ep_*}\in A_f\ \ {\rm and}\ \ b_*\leq b_{\ep_*},\ b_*\not=b_{\ep_*}.
\]
Recalling that $d b_{\ep}(\cdot)/d\ep=\chi_{\{x\mid b(x)<\ep\}}(\cdot)$ and also
Proposition \ref{pr:derivative of c(b)}, we get that
\[
\begin{split}
\frac{d}{d\epsilon}c^*(b_{\ep})&=
\frac{1}{\lambda(b_{\ep})}
\Big\langle\psi_{b_{\ep}}\widetilde{\psi}_{b_{\ep}},\frac{d b_{\ep}(\cdot)}{d\ep}\Big\rangle
\geq0,\,\not\equiv0\ \ {\rm for}\ \ \ep\in[0,\ep_*].\\
&\therefore c^*(b_*)<c^*(b_{\ep_*})\leq\sup_{b\in A_f}\,c^*(b)=c^*(b_*).
\end{split}
\]
This contradiction completes the proof.
\end{proof}

\section{Derivation of the Euler--Lagrange equation under $L^2$ constraints}
In this section, we consider a particular case when $f(b)=b^2$.
The next lemma plays a key role in deriving the Euler--Lagrange equation which the maximizer of
$c^*( b)$ satisfies under $L^2$ constraint $\frac{1}{L}\int_0^L b^2 dx=\beta$.
\begin{lemma}\label{lm:critical point}
Let
$f$
be a smooth function with
$f'(b)>0$ $(b>0)$ and $f'(0)=0$.
Then, any maximizer
$b_*\in A_f$
of
$c^*(b)$
satisfies:
\begin{equation}\label{eq:critical}
\psi_{b_*}\widetilde{\psi}_{b_*}=f'(b_*)\Big/\int_0^L f'(b_*)d x,
\end{equation}
where
$\psi_b=\psi_{\lambda(b),b}$,
$\widetilde{\psi}_b=\widetilde{\psi}_{\lambda(b),b}$.
\end{lemma}
\begin{proof}
Extend the domain of
$f$ to $\R$
as follows:
\[
f:{\mathbb R}\ni b\mapsto f(|b|)\in[0,\infty).
\]
Let
$I(b)=\int_0^L f(b) d x$,
we recall that
$c^*(b)\leq c^*(|b|)$ from Nadin's formula.
Hence, it is clear that
any maximizer
$b_*\in A_f$
of
$c^*(b)$
is a maximizer of
$c^*(b)$
in
$\overline{A}_f
=\{b\mid I(b)=\beta L,\ \int_0^L b d x>0\}$.

Fix an $L$-periodic function
$\varphi\in C^{\infty}({\mathbb R})$
arbitrarily and let
$b_*\in A_f$
be a maximizer of
$c^*(b)$.
Then,
for any sufficiently small
$\ep\in {\mathbb R}$,
\[
c^*(b_*)\geq c^*(\mu_{\varphi}(\ep)b_*+\ep\varphi),
\]
where the smooth function
$\ep\mapsto\mu_{\varphi}(\ep)$
is uniquely determined by the implicit function theorem with conditions
$I(\mu b_*+\ep\varphi)=\beta L$ and
$\mu_{\varphi}(0)=1$

Therefore, we have
\[
\begin{split}
&0=\frac{d}{d\epsilon}c^*(\mu_{\varphi}(\ep)b_*+\ep\varphi)
\Big|_{\ep=0}
=\partial c^*(b_*)[\mu_{\varphi}'(0)b_*+\varphi]\ \mbox{and}\ \ 
0=\partial I(b_*)[\mu_{\varphi}'(0)b_*+\varphi].\\
&\therefore \
\mu_{\varphi}'(0)=-\partial I(b_*)[\varphi]/\partial I(b_*)[b_*].\\
&\therefore \
0=\partial c^*(b_*)[\varphi]+C\,\partial I(b_*)[\varphi]\ \
(C=-\partial c^*(b_*)[b_*]/\partial I(b_*)[b_*]).
\end{split}
\]
Furthermore, we compute that
\[
\partial c^*(b_*)[\varphi]
=(-1/\lambda(b_*))\langle\psi_{b_*}\widetilde{\psi}_{b_*},\varphi\rangle,\
\partial I(b_*)[\varphi]=\langle f'(b_*),\varphi\rangle.
\]
Hence, there exists a constant
$\widetilde{C}$
such that
\[
\langle\psi_{b_*}\widetilde{\psi}_{b_*}+\widetilde{C}f'(b_*),
\varphi\rangle=0.
\]
Since
$\varphi$
is chosen arbitrarily, we have
\[
\psi_{b_*}\widetilde{\psi}_{b_*}+\widetilde{C}f'(b_*)=0.
\]
By integrating both sides of this equation,
we obtain
\[
\widetilde{C}
=-\int_0^L \psi_{b_*}\widetilde{\psi}_{b_*} d x\Big/\int_0^L f'(b_*) d x
=-1\Big/\int_0^L f'(b_*)d x.
\]
This completes the proof.
\end{proof}

We can now go back to the proof of Theorem \ref{th:Euler--Lagrange}.
\begin{proof}
[Proof of Theorem \ref{th:Euler--Lagrange}]
Let
$b$
be a maximizer of
$c^*(\cdot)$
in
$A_2$. For the convenience, in this proof, we note
$\lambda=\lambda(b)$,
$k=k(\lambda,b)$,
$\psi=\psi_{b}$
and
$\widetilde{\psi}=\widetilde{\psi}_{b}$.

If we let
\[
\psi=e^{\varphi}\phi,\ \widetilde{\psi}
=e^{-\varphi}\phi,
\]
namely,
\[
\phi=\sqrt{\psi\widetilde{\psi}},\ \varphi
=\frac{1}{2}\log(\psi/\widetilde{\psi}),
\]
then by substituting
$\psi=e^{\varphi}\phi$,
$\widetilde{\psi}=e^{-\varphi}\phi$
into
\[
\begin{split}
&-\psi''-2\lambda\psi'-(\lambda^{2}+b)\psi=k\psi,\\
&-\widetilde{\psi}''
+2\lambda\widetilde{\psi}'
-(\lambda^{2}+b)\widetilde{\psi}=k\widetilde{\psi},
\end{split}
\]
respectively, we obtain that
\[
\begin{split}
&-(\phi''+2\varphi'\phi'+\varphi^{\prime2}\phi+\varphi''\phi)
-2\lambda\,(\phi'+\varphi'\phi)-(\lambda^{2}+b)\phi=k\phi,\\
&-(\phi''-2\varphi'\phi'+\varphi^{\prime2}\phi-\varphi''\phi)
+2\lambda\,(\phi'-\varphi'\phi)-(\lambda^{2}+b)\phi=k\phi.
\end{split}
\]
Thus, it follows immediately that:
\begin{equation}\label{eq:phi-varphi}
\left\{
\begin{split}
&2\varphi'\phi'+\varphi''\phi+2\lambda\phi'=0,\\
&-\phi''-\varphi^{\prime2}\phi-2\lambda\varphi'\phi-(\lambda^{2}+b)\phi
=k\phi.
\end{split}
\right.
\end{equation}
Multiplying both sides of this first equality by
$\phi$
and integrating, we have
\[
\phi^{2}\varphi'+\lambda\phi^{2}=C_{1}\ \
\therefore \varphi'=C_2\phi^{-2}-\lambda,
\]
where
$C_2=\lambda L/\int_{0}^{L}\phi^{-2}$ is determined by the constraint.

Substituting
$\varphi'=C_2\phi^{-2}-\lambda$
into the second equality of \eqref{eq:phi-varphi}, we have
\[
-\phi''-C_2^{2}\phi^{-3}-(b+k)\phi=0.
\]
Thus by \eqref{eq:critical}
and
$\psi\widetilde{\psi}=\phi^2$,
if we put
$C_3=2\int_0^L b d x$,
then
\[
-\phi''-C_2^{2}\phi^{-3}-\frac{C_3}{2} \phi^3-k\phi=0.
\]
Multiplying both sides of this equality by
$\phi'$
and integrating,
\[
-\phi^{\prime2}+C_2^{2}\phi^{-2}
-\frac{C_3}{4}\phi^4-k\phi^{2}
=C.
\]
Substituting
$\phi=\sqrt{2b/C_3}$
into this equality,
\[
-\frac{b^{\prime 2}}{2C_3b}
+\frac{C_2^2C_3}{2b}-\frac{b^2}{C_3}
-\frac{2k b}{C_3}=C.\ \ 
\therefore\
b^{\prime2}-C_2^2C_3^{2}
+2b^3
+4k b^{2}
=-2CC_3b.
\]
Differentiating both sides with respect to
$x$
and dividing by
$b'(x)>0$ $(x\in(0,L/2))$, we get the Euler--Lagrange equation as follow:
\[
b''+4k b+3b^2=-CC_3,\ \ 
-CC_3=\frac{4k}{L}\int_0^L b dx+3\beta.
\]
The proof is completed.
\end{proof}
\section{Second variation of $c^*(b)$ under constraints}
In this section, by calculating the second variation
of $c^*(b)$ around the constant function $b_0:=f^{-1}(\beta)$,
we prove a sufficient condition and a necessary condition under which
$b_0$ is a local maximizer
of $c^*(b)$ in $A_f$ under $L^{\infty}$-perturbation. Moreover, we also prove that $b_0:=\sqrt{\beta}$
is a local maximizer of $c^*(b)$ in $A_2$ under $L^2$-perturbation.
\begin{lemma}[Second variation of $c^*(b)$]
\label{lm:second derivative of c(b)}
Define $b_0:=f^{-1}(\beta)$.
Then we find
\begin{equation}
\label{eq:second derivative of c(b)}
\partial^2 c^*(b_0)[v]^2=\frac{-1}{f^{-1}(\beta)^{1/2}}\partial_b^2 k(\lambda(b_0),b_0)[v]^2
-\frac{1}{2L^2 f^{-1}(\beta)^{3/2}}\Big(\int_0^L v d x\Big)^{2}.
\end{equation}
\end{lemma}
\begin{proof}
Let
$b_0:=f^{-1}(\beta)$.
Following from the direct calculation,
\[
\psi_{\lambda,b_0}=\widetilde{\psi}_{\lambda,b_0}
=1/\sqrt{L},\ \
-
k(\lambda,b_0)=\lambda^2+f^{-1}(\beta),\ \
\lambda(b_0)=f^{-1}(\beta)^{1/2}.
\]
Hence by \eqref{eq:derivative of k} and \eqref{eq:derivative of lambda}, we obtain that
\[
\begin{split}
\partial_b k(\lambda,b_0)[v]
&=-\langle
\psi_{\lambda,b_0}\widetilde{\psi}_{\lambda,b_0},
v\rangle
=-\frac{1}{L}\int_0^L v d x,\ \
\partial_{\lambda,b}^{2}k(\lambda,b_0)
=0,\\
\partial\lambda(b_0)[v]
=&\frac{\partial_{b}k(\lambda(b_0),b_0)[v]
-\lambda(b_0)\partial_{\lambda,b}^{2}k(\lambda(b_0),b_0)[v]}
{\lambda(b_0)\partial_{\lambda}^{2}k(\lambda(b_0),b_0)}
=\frac{\int_0^L v d x}{2f^{-1}(\beta)^{1/2}L}.
\end{split}
\]
As a consequence of \eqref{eq:derivative of c(b)}, we get that
\[
\begin{split}
\partial^{2}c^*(b_0)[v]^{2}
=&\partial\Big(\frac{-\partial_{b}k(\lambda(b),b)[v]}{\lambda(b)}
\Big)_{b=b_0}[v]\\
=&\frac{-\partial_{b}^{2}k(\lambda(b_0),b_0)[v]^{2}}{\lambda(b_0)}\\&
-\Big(\frac{\partial^{2}_{\lambda,b}k(\lambda(b_0),b_0)[v]}
{\lambda(b_0)}
-\frac{\partial_{b}k(\lambda(b_0),b_0)[v]}
{\lambda(b_0)^{2}}\Big)\partial\lambda(b_0)[v]\\
=&\frac{-\partial_{b}^{2}k(\lambda(b_0),b_0)[v]^{2}}
{f^{-1}(\beta)^{1/2}}
-\frac{1}{2f^{-1}(\beta)^{3/2}L^{2}}
\Big(\int_0^L v d x\Big)^{2}.
\end{split}
\]
\end{proof}
\begin{lemma}[Second variation of $k(\lambda,b)$]
\label{lm:partial^2_b k}
Let
\[
v=u_0+\sum_{n=1}^{\infty}(u_n\varphi_n+v_n\psi_n)
\]
be a bounded function, where $u_{j-1},v_j\in\R$ $(j\in\N)$,
\[
\varphi_n(x):=\sqrt{2/L}\cos{\frac{2\pi n x}{L}},\ \
\psi_n(x):=\sqrt{2/L}\sin{\frac{2\pi n x}{L}}.
\]
Then we have:
\begin{equation}\label{eq:partial^2_b k}
\partial^2_b k(\lambda(b_0),b_0)[v]^2
=\frac{-1}{2L}\sum_{n=1}^{\infty}\frac{u_n^2+v_n^2}{(n\pi/L)^2+f^{-1}(\beta)}.
\end{equation}
\end{lemma}
\begin{proof}
From \eqref{eq:derivative of k}, we know
\[
\begin{split}
\partial^2_b k(\lambda(b_0),b_0)[v]^2
&
=-\partial_b(\langle\psi_{\lambda,b}\widetilde{\psi}_{\lambda,b},v\rangle)_{
\lambda=\lambda(b_0),b=b_0}[v]\\
&
=-\langle\psi
D\widetilde{\psi}
+D\psi\widetilde{\psi},
v\rangle,
\end{split}
\]
where
$\psi=\psi_{\lambda(b_0),b_0}$,
$\widetilde{\psi}=\widetilde{\psi}_{\lambda(b_0),b_0}$,
$D\psi=\partial_b\psi_{\lambda(b_0),b_0}[v]$
and
$D\widetilde{\psi}=\partial_b\widetilde{\psi}_{\lambda(b_0),b_0}[v]$.

Define
$\varphi_n(x):=\sqrt{2/L}\cos(2\pi n/L)x$,
$\psi_n(x):=\sqrt{2/L}\sin(2\pi n/L)x$
and
\[
v=u_0+\overset{\infty}
{\underset{n=1}{\sum}}
(u_n\varphi_n+v_n\psi_n)
\quad
(u_{j-1},\,v_j\in{\mathbb R},\,j\in{\mathbb N}).
\]
We recall that $\psi=\widetilde{\psi}\ (=1/\sqrt{L})$.
Differentiating both sides of the following equations
with respect to $b$,
\[
-L_{\lambda,b_0}\psi=k(\lambda,b_0)\psi,\
-L_{\lambda,b_0}^{*}\widetilde{\psi}
=k(\lambda(b_0),b_0)\widetilde{\psi},\
\|\psi\|^2=\langle\widetilde{\psi},\psi\rangle=1,
\]
we obtain the following result by recalling that $\psi=\widetilde{\psi}\ (=1/\sqrt{L})$.
\[
\begin{split}
&-L_{\lambda,b_0}D\psi-\psi v
=k(\lambda,b_0)D\psi +\partial_b k(\lambda,b_0)[v]\psi,\\
&-L_{\lambda,b_0}^{*}D\widetilde{\psi}
-\widetilde{\psi} v
=k(\lambda,b_0)D\widetilde{\psi}
 +\partial_b k(\lambda,b_0)[v]\widetilde{\psi},\\
&\langle D\widetilde{\psi},\widetilde{\psi}\rangle
=\langle D\widetilde{\psi},\psi\rangle
=-\langle\widetilde{\psi},D\psi\rangle
=-\langle\psi,D\psi\rangle=0.
\end{split}
\]
Therefore, since $\partial_b k(\lambda,b_0)[v]=-u_0$, we have:
\[
\begin{split}
&-(L_{\lambda(b_0),b_0}+k(\lambda(b_0),b_0))D\psi
=(1/\sqrt{L})(v-u_0),\
\langle\widetilde{\psi},D\widetilde{\psi}\rangle=0,\\
&-(L_{\lambda(b_0),b_0}+k(\lambda(b_0),b_0))
D\widetilde{\psi}
=(1/\sqrt{L})(v-u_0),
\langle\psi,D\psi\rangle=0.
\end{split}
\]
\[
\begin{split}
\therefore
D\psi=
\sum_{n=1}^{\infty}
\frac{1}{n\pi\sqrt{L}}
\Big\{&
\frac{n\pi u_n+f^{-1}(\beta)^{1/2}L v_n}
{(2n\pi/L)^2+4f^{-1}(\beta)}\varphi_n
\\&
+\frac{-Lf^{-1}(\beta)^{1/2}u_n+n\pi v_n}
{(2n\pi/L)^2+4f^{-1}(\beta)}\psi_n\Big\},\\
D\widetilde{\psi}=
\sum_{n=1}^{\infty}
\frac{1}{n\pi\sqrt{L}}
\Big\{&
\frac{n\pi u_n-f^{-1}(\beta)^{1/2}L v_n}
{(2n\pi/L)^2+4f^{-1}(\beta)}\varphi_n
\\&
+\frac{Lf^{-1}(\beta)^{1/2}u_n+n\pi v_n}
{(2n\pi/L)^2+4f^{-1}(\beta)}\psi_n\Big\}.
\end{split}
\]
\[
\therefore
\partial^2_b k(\lambda(b_0),b_0)[v]^2
=-\langle\widetilde{\psi}D\psi+D\widetilde{\psi}\psi,
v\rangle
=
\frac{-1}{2L}\sum_{n=1}^{\infty}
\frac{u_n^2+v_n^2}
{(n\pi/L)^2+f^{-1}(\beta)}.
\]
\end{proof}
In order to the prove Theorem \ref{th:local maximality L^{infty}},
we need state the following Lemma \ref{lm:projection} at first. The prove will be given in the appendix.
\begin{lemma}[Projection operator]
\label{lm:projection}
Define $b_0:=f^{-1}(\beta)$, $\ep_0:=\|b_0\|_{\infty}=f^{-1}(\beta)$ and
\[
B_{\ep_0}:=\Big\{b_0+v\,\Big|\,
v\in L^{\infty}_{L\mbox{\rm\tiny-per}},\ \int_0^L v dx=0,\ \|v\|_{\infty}<\ep_0
\Big\}.
\]
Let $P$ be a map defined as follows:
\[
B_{\ep_0}\ni b\mapsto P(b):=\mu(b)b\in A_f,
\]
where $\mu=\mu(b)$ is a real number with
$\int_0^L f(\mu b)dx=\beta L$. Then
\begin{equation}\label{eq:projection}
P(b_0+v)=
b_0+v-\frac{f''(b_0)}{2f'(b_0)L}\|v\|_2^2+o(\|v\|_2^2)\ \
{\rm as}\ \ B_{\ep_0}\ni b_0+v\rightarrow b_0\ \ {\rm in}\ \ L^{\infty}.
\end{equation}
\end{lemma}
\begin{proof}[of Theorem \ref{th:local maximality L^{infty}}]
Let $b_0$, $\ep_0$, $P$ be defined as in Lemma \ref{lm:projection}. Therefore,
for any $\ep\in(0,\ep_0]$ there exists an $\eta(\ep)>0$ such that
\[
\{b\in A_f\mid \|b-b_0\|_{\infty}\leq\eta(\ep)\}
\subset P(B_{\ep}).
\]
By using Lemmas \ref{pr:derivative of c(b)},
\ref{lm:second derivative of c(b)}, \ref{lm:partial^2_b k} and \ref{lm:projection},
for $b_0+v\in B_{\ep_0}$, we obtain
\[
\begin{split}
&c^*(P(b_0+v))=c^*\big(b_0+v-(f''(b_0)/(2f'(b_0)L))\|v\|_2^2+o(\|v\|_2^2)\big)\\
&=c^*(b_0)+\partial c^*(b_0)\big[v-(f''(b_0)/(2f'(b_0)L))\|v\|_2^2\big]\\
&+\frac{1}{2}\partial^2 c^*(b_0)[v]^2
+o(\|v\|_2^2)\\
&=c^*(b_0)-\frac{f''(b_0)}{2f^{-1}(\beta)^{1/2}f'(b_0)L}
\sum_{n=1}^{\infty}(u_n^2+v_n^2)\\
&\hspace{40pt}+
\frac{1}{4f^{-1}(\beta)^{1/2}L}\sum_{n=1}^{\infty}\frac{u_n^2+v_n^2}{(n\pi/L)^2+f^{-1}(\beta)}
+o(\|v\|_2^2).
\end{split}
\]
\begin{equation}\label{eq:expansion of c}
\begin{split}
\therefore\ & c^*(P(b_0+v))
=c^*(b_0)\\
&-\frac{f''(b_0)}{2f^{-1}(\beta)^{1/2}L}
\sum_{n=1}^{\infty}
\frac{\Big(\frac{n\pi}{L}\Big)^2+f^{-1}(\beta)-\frac{f'(b_0)}{2f''(b_0)}}
{f'(b_0)\{(n\pi/L)^2+f^{-1}(\beta)\}}(u_n^2+v_n^2)+o(\|v\|_2^2).
\end{split}
\end{equation}
where $u_n$ and $v_n$ are determined by
\[
u_n=\langle\varphi_n, v\rangle,\ v_n=\langle\psi_n,v\rangle
\]
and $\varphi_n$, $\psi_n$ are eigenfunctions defined in Lemma \ref{lm:partial^2_b k}.

Hence, if the following inequality holds,
\[
2f''(f^{-1}(\beta))(\pi^2/L^2+f^{-1}(\beta))-f'(f^{-1}(\beta))>0
\]
then we can claim:
\[
(\pi/L)^2+f^{-1}(\beta)-\frac{f'(b_0)}{2f''(b_0)}>0
\]
and for sufficiently small $\ep>0$,
\[
c^*(b_0)>c^*(P(b_0+v))\ \ {\rm where}\ \ b_0+v\in B_{\ep}\backslash\{b_0\}.
\]
Thus, let $\eta=\eta(\ep)$, we get
\[
c^*(b_0)>c^*(b)\ \ {\rm for}\ \ b\in A_f\ \ {\rm with}\ \
0<\|b-b_0\|_{\infty}\leq\eta.
\]
Therefore, we claim that $b_0$ is a local maximizer of $c^*(b)$ in $A_f$.

Next, we finish the proof by assuming
\[
2f''(f^{-1}(\beta))(\pi^2/L^2+f^{-1}(\beta))-f'(f^{-1}(\beta))<0
\]
then by \eqref{eq:expansion of c}, it is clear that
\[
\left.\frac{d}{d t}c^*(P(b_0+t \varphi_1))\right|_{t=0}=0,\ \
\left.\frac{d^2}{d t^2}c^*(P(b_0+t \varphi_1))\right|_{t=0}
>0.
\]
Hence, for any sufficiently small $t$,
\[
c^*(b_0)<c^*(P(b_0+t\varphi_1)).
\]
On the other hand, there exists a large enough $N$ such that
\[
2f''(f^{-1}(\beta))((N\pi/L)^2+f^{-1}(\beta))-f'(f^{-1}(\beta))>0.
\]
Thus, by recalling \eqref{eq:expansion of c}, we have:
\[
\left.\frac{d}{d t}c^*(P(b_0+t \varphi_N))\right|_{t=0}=0,\ \
\left.\frac{d^2}{d t^2}c^*(P(b_0+t \varphi_N))\right|_{t=0}
<0.
\]
Hence for any sufficiently small $t$,
\[
c^*(b_0)>c^*(P(b_0+t\varphi_N)).
\]
Therefore $b_0$ is neither a local maximizer nor a local minimizer of
$c^*(b)$ in $A_f$.
This completes the proof.
\end{proof}

Corollary \ref{cr:L^p} can be proved immediately by applying the result of Theorem \ref{th:local maximality L^{infty}}
\begin{proof}[Proof of Corollary \ref{cr:L^p}]
For the case of $f(b)=b^p$ $(p\in(1,\infty))$,
\[
\begin{split}
2f''(f^{-1}&(\beta))(\pi^2/L^2+f^{-1}(\beta))-f'(f^{-1}(\beta))\\
=&2p(p-1)\beta^{(p-2)/p}(\pi^2/L^2+\beta^{1/p})-p\beta^{(p-1)/p}\\
=&p\beta^{(p-2)/p}\{2(p-1)\pi^2/L^2+(2p-3)\beta^{1/p}\}.
\end{split}
\]
Hence, the condition $D(\beta)>0$
is equivalent to
$p\geq3/2$ or
\[
1<p<3/2\ \ {\rm and}\ \
\beta^{1/p}<\frac{2(p-1)\pi^2}{(3-2p)L^2}.
\]
The proof is completed.
\end{proof}
\begin{proof}[Proof of Theorem \ref{th:local minimality L^2}]
Let $b_0:=\sqrt{\beta}$ and $P$
be a map from $B_{\ep}:=\{b_0+v\mid \|v\|_2<\ep,\ \int_0^L v dx=0\}$
into $A_2$ as follows:
\[
P(b_0+v)
:=\frac{\sqrt{\beta L}}{2\|b_0+v\|_{2}}(b_0+v)
=b_0+v-\frac{\|v\|_{2}^2}{\|b_0\|_{2}^2}b_0+O(\|v\|_{2}^3).
\]
Then there exists $\eta(\ep)>0$ such that
\[
\{b\in A_2\mid\|b-b_0\|_2<\eta(\ep)\}\subset P(B_{\ep}).
\]
Moreover,
by Lemmas \ref{pr:derivative of c(b)}, \ref{lm:second derivative of c(b)} and \ref{lm:partial^2_b k}, for any $b_0+v\in B_{\ep}$, we get:
\[
\begin{split}
&c^*(P(b_0+v))=c^*\big(b_0+v-(\|v\|_{2}^2/(2\|b_0\|_{2}^2)) b_0+O(\|v\|_{2}^3)\big)\\
&=c^*(b_0)+\partial c^*(b_0)\big[v-(\|v\|_{2}^2/(2\|b_0\|_{2}^2)) b_0\big]+\frac{1}{2}\partial^2 c^*(b_0)[v]^2
+O(\|v\|_2^3)\\
&=c^*(b_0)-\frac{1}{2\beta^{3/4}L}\sum_{n=1}^{\infty}(u_n^2+v_n^2)\\&\hspace{40pt}
+\frac{1}{4\beta^{1/4}L}\sum_{n=1}^{\infty}\frac{u_n^2+v_n^2}{(n\pi/L)^2+\beta^{1/2}}
+O(\|v\|_2^3)\\
&=c^*(b_0)-\frac{1}{4\beta^{3/4}L}\sum_{n=1}^{\infty}
\frac{2(n\pi/L)^2+\beta^{1/2}}{(n\pi/L)^2+\beta^{1/2}}(u_n^2+v_n^2)
+O(\|v\|_2^3),
\end{split}
\]
where
\[
u_n:=\langle v,\varphi_n\rangle,\ v_n:=\langle v,\psi_n\rangle
\]
and $\varphi_n$, $\psi_n$\ $(n\in\N)$ are defined as in Lemma \ref{lm:partial^2_b k}.
Thus there exists an $\ep>0$ such that
\[
c^*(b_0)>c^*(P(b_0+v))\ \ {\rm for}\ \ b_0+v\in B_{\ep}\backslash\{b_0\}.
\]
Therefore we can claim that exists an $\eta=\eta(\ep)$ such that
\[
c^*(b_0)>c^*(b)\ \ {\rm for}\ \ b\in
\{\tilde{b}\in A_2\mid0<\|\tilde{b}-b_0\|_2<\eta\}\,(\subset P(B_{\ep})\backslash\{b_0\}).
\]
This completes the proof.
\end{proof}
\section{Asymptotic analysis of the maximizers
with respect to the period $L$ and the mass $\beta$}
In this section, we consider asymptotic behavior of the maximizers
under $L^2$ constraint $\frac{1}{L}\int_0^L b^2 dx=\beta$
when the period $L$ or mass $b$ is very small.
We also deal with the variational problem for a class of
step functions under $L^p$ constraint $\frac{1}{L}\int_0^Lb^pdx=\beta$
when the period $L$ or the mass $\beta$ is very large.
\begin{lemma}\label{lm:estimate lambda, k}
Let $\lambda(b)$ be the value corresponding to each given $b$, such that
\[
\frac{-k(\lambda(b),b)}{\lambda(b)}=c^*(b)\,
\Big(=\min_{\lambda>0}\frac{-k(\lambda,b)}{\lambda}\Big).
\]
Then, we have the following estimate:
\begin{equation}\label{eq:estimate lambda}
\sqrt{\alpha+\alpha^2L^2}-\alpha L\leq\lambda(b)\leq \sqrt{\alpha+\alpha^2L^2}-\alpha L,
\end{equation}
where $\alpha=\frac{1}{L}\int_0^L b dx$.
\end{lemma}
\begin{proof}
Define $\alpha:=\frac{1}{L}\int_0^L b dx$.
In Lemma 4.3 of Liang, Lin and Matano \cite{LLM}, It has already been proved that
\begin{equation}\label{eq:estimate k}
\alpha +\lambda^2\leq -k(\lambda,b)\leq \alpha+\alpha^2L^2+\lambda^2.
\end{equation}
Hence, it follows that:
\[
\min_{\lambda>0}\frac{\alpha +\lambda^2}{\lambda}
\leq c^*(b)=\min_{\lambda>0}\frac{-k(\lambda,b)}{\lambda}
\leq\min_{\lambda>0}\frac{\alpha+\alpha^2L^2+\lambda^2}{\lambda}.
\]
\begin{equation}\label{eq:estimate c(b)}
\therefore2\sqrt{\alpha}\leq c^*(b)\leq 2\sqrt{\alpha+\alpha^2L^2}.
\end{equation}
By the definition of $\lambda(b)$ and estimates \eqref{eq:estimate k} and \eqref{eq:estimate c(b)}, we have
\[
\begin{split}
\alpha +\lambda(b)^2&\leq c^*(b)\lambda(b)=-k(\lambda(b),b)
\leq2\sqrt{\alpha+\alpha^2L^2}\lambda(b).\\
&\therefore \lambda(b)^2-2\sqrt{\alpha+\alpha^2L^2}\lambda(b)+\alpha\leq0.
\end{split}
\]

Therefore, we obtain that
\[
\sqrt{\alpha+\alpha^2L^2}-\alpha L\leq\lambda(b)\leq\sqrt{\alpha+\alpha^2L^2}+\alpha L.
\]
The proof is completed.
\end{proof}
\begin{corollary}\label{cr:estimate k(lambda(b),b)}
Define $\alpha:=\frac{1}{L}\int_0^L b dx$. Then
\begin{equation}\label{eq:estimate k(lambda(b),b)}
\alpha+\big(\sqrt{\alpha+\alpha^2L^2}-\alpha L\big)^2
\leq -k(\lambda(b),b)\leq \alpha+\alpha^2L^2
+\big(\sqrt{\alpha+\alpha^2L^2}+\alpha L\big)^2.
\end{equation}
\end{corollary}
\begin{proof}
This inequality is an obvious consequence of \eqref{eq:estimate k} and \eqref{eq:estimate lambda}.
\end{proof}
\begin{lemma}\label{lm:norm ineq}
For any nonnegative function $b\in C^2_{L\mbox{\tiny-per}}$ satisfying $\int_0^L b^2 dx=\beta$, we have the following estimate:
\[
\big\|b-\sqrt{\beta}\big\|_{L^{\infty}}\leq L\|b'\|_{L^{\infty}}
\leq L\|b''\|_{L^1([0,L])}\leq L^2\|b''\|_{L^{\infty}}
\]
\end{lemma}
\begin{proof}
By the continuity of $b\leq0 $ and $\int_0^L b^2dx=\beta$,
there exists a point $x_1\in[0,L)$ such that $b(x_1)=\sqrt{\beta}$. Hence
for any $x\in[0,L)$,
\[
\begin{split}
\big|b(x)-\sqrt{\beta}\big|&=|b(x)-b(x_1)|\leq\int_0^L |b'|dx\leq L\|b'\|_{L^{\infty}}.\\
&\therefore\big\|b-\sqrt{\beta}\big\|_{L^{\infty}}\leq L\|b'\|_{L^{\infty}}.
\end{split}
\]
By the smoothness and periodicity of $b$, exists a point $x_2\in[0,L)$ satisfying $b'(x_2)=0$.
Hence for any $x\in[0,L)$, the following estimate holds:
\[
\begin{split}
\big|b'(x)|&=|b'(x)-b'(x_2)|\leq\int_0^L |b''|dx\leq L\|b''\|_{L^{\infty}}.\\
&\therefore\|b'\|_{L^{\infty}}\leq\|b''\|_{L^1([0,L])}\leq L\|b''\|_{L^{\infty}}.
\end{split}
\]
The proof is completed.
\end{proof}
\subsection{The case $L\rightarrow0$}

\begin{proof}[Proof of Theorem \ref{th:L->0}]
Let
\[
\widetilde{b}(x)=\widetilde{b}(L,x)=b_L(L x).
\]
Then for each $L$, $\widetilde{b}$ satisfies
\begin{equation}\label{eq:b(L x)}
\frac{\widetilde{b}''}{L^2}+4k\widetilde{b}+3\widetilde{b}^2
=4k\int_0^1\widetilde{b}dx+3\beta\ \ {\rm in}\ \ \R/\Z,\ \
\int_0^1\widetilde{b}^2dx=\beta,
\end{equation}
where $k=k(\lambda(b_L),b_L)$.
From Corollary \ref{cr:estimate k(lambda(b),b)}, we get
\[
\lim_{L\rightarrow0}k(\lambda(b_L),b_L)=-2\alpha\ \
\Big(\alpha:=\frac{1}{L}\int_0^L b dx\Big).
\]
Hence, by recalling \eqref{eq:b(L x)},
\[
\big\|\widetilde{b}''\big\|_{L^1([0,1])}\leq
L^2\,\big(8|k|\big\|\widetilde{b}\big\|_{L^1([0,1])}+6\beta\big).
\]
Therefore, Lemma \ref{lm:norm ineq} implies that:
\[
\big\|\widetilde{b}-\sqrt{\beta}\big\|_{L^{\infty}}
\leq \big\|\widetilde{b}'\big\|_{L^{\infty}}\leq \big\|\widetilde{b}''\big\|_{L^1([0,1])}=O(L^2)\ \
{\rm as}\ \ L\rightarrow0.
\]
Thus by \eqref{eq:b(L x)}, we have
\begin{equation}\label{eq:L-infty b''}
\big\|\widetilde{b}''\big\|_{L^{\infty}}\leq L^2\,
\big(8|k|\big\|\widetilde{b}\big\|_{L^{\infty}}
+3\big\|\widetilde{b}^2\big\|_{L^{\infty}}+3\beta\big).
\end{equation}
Define $v:=\widetilde{b}-\sqrt{\beta}$. Also since \eqref{eq:b(L x)},
\begin{equation}\label{eq:v}
\left\{
\begin{split}
&\frac{v''}{L^2}+4kv+6\sqrt{\beta}v+3v^2=4k\int_0^1 v dx\ \ {\rm in}\ \ \R/\Z,\\
&2\sqrt{\beta}\int_0^1 v dx+\int_0^1 v^2 dx=0.
\end{split}
\right.
\end{equation}
From Lemma \ref{lm:norm ineq} and \eqref{eq:L-infty b''}, we claim that $\|v\|_{C^2}=O(L^2)$.

Next we only need to prove that for each $k\in\N$,
\begin{equation}\label{eq:induction L->0}
\|v\|_{C^2}=O(L^{2k})\Rightarrow\|v\|_{C^2}=O(L^{2(k+1)}),
\end{equation}
since the statement can be proved by induction.
If we assume
$\|v\|_{C^2}=O(L^{2k})$, then by \eqref{eq:v}, we find
\[
\|v''\|_{L^{\infty}}\leq L^2\,(2(4|k|+3\sqrt{\beta})\|v\|_{L^{\infty}}+3\|v^2\|_{L^{\infty}})
=O(L^{2(k+1)}).
\]
Hence, by applying Lemma \ref{lm:norm ineq},
\[
\|v\|_{L^{\infty}}\leq\|v'\|_{L^{\infty}}\leq\|v''\|_{L^{\infty}}=O(L^{2(k+1)}).
\ \ \therefore\ \|v\|_{C^2}=O(L^{2(k+1)})
\]
and \eqref{eq:induction L->0} is proved. This completes the proof.
\end{proof}
\subsection{The case $\beta\rightarrow0$}
Let $\widetilde{b}(x)=L^2b(Lx)$. Then a simple rescaling argument gives that:
\begin{equation}\label{eq:rescaling}
\ k(\lambda,b)=\frac{k(\lambda L,\widetilde{b})}{L^2}.
\end{equation}
Hence, it follows immediately that
\[
c^*(b)=\min_{\lambda>0}\,\frac{-k(\lambda,b)}{\lambda}
=\min_{\lambda>0}\,\frac{-k(\lambda L,\widetilde{b})}{\lambda L^2}
=\frac{c^*(\widetilde{b})}{L}.
\]
Thus, we claim that $b$ is a maximizer of $c^*(\cdot)$ in $A_{2,L,\beta}$ if and only if
$\widetilde{b}$ is a maximizer of $c^*(\cdot)$ in $A_{2,1,L^4\beta}$.
Therefore, in the following proof, we may assume that the period $L=1$.

Now we go back to the proof of Theorem \ref{th:beta->0}.
\begin{proof}[Proof of Theorem \ref{th:beta->0}]
Let
\[
\widetilde{b}(x)=\widetilde{b}(\beta,x):=\frac{b_{\beta}}{\sqrt{\beta}}.
\]
Then for each $\beta>0$, $\widetilde{b}=\widetilde{b}(\beta,\cdot)$ satisfies
\begin{equation}\label{eq:b/beta^{1/2}}
\left\{
\begin{split}
&\widetilde{b}''+4k\widetilde{b}+3\sqrt{\beta}\widetilde{b}^2
=4k \int_0^1 \widetilde{b}dx+3\sqrt{\beta}
\ \ {\rm in}\ \ \R/\Z,\\
&\int_0^1 \widetilde{b}^2 dx=1.
\end{split}
\right.
\end{equation}
Using \eqref{eq:estimate k(lambda(b),b)} and recalling $\int_0^1 b dx\leq\sqrt{\beta}$, we get
\[
k=k(\lambda(b_{\beta}),b_{\beta})=O(\sqrt{\beta}).
\]
Hence by \eqref{eq:b/beta^{1/2}}, it is clear that
\[
\big\|\widetilde{b}''\big\|_{L^1([0,1])}
\leq 8|k|\big\|\widetilde{b}\big\|_{L^1([0,1])}+3\sqrt{\beta}(\beta+1)
=O(\sqrt{\beta}).
\]
Using Lemma \ref{lm:norm ineq},
\[
\big\|\widetilde{b}-1\big\|_{L^{\infty}}\leq \big\|\widetilde{b}'\big\|_{L^{\infty}}
\leq \big\|\widetilde{b}''\big\|_{L^1([0,1])}=O(\sqrt{\beta}).
\]
Hence by \eqref{eq:b/beta^{1/2}},
\begin{equation}\label{eq:L^{infty} b'' 2}
\big\|\widetilde{b}''\big\|_{L^{\infty}}=O(\sqrt{\beta}).
\end{equation}
Put $v:=\widetilde{b}-1$. Then by \eqref{eq:b/beta^{1/2}},
\begin{equation}\label{eq:v 2}
\left\{
\begin{split}
&v''+4k v+6\sqrt{\beta}v+3\sqrt{\beta}v^2=4k\int_0^1 v
\ \ {\rm in}\ \ \R/\Z,\\
&\int_0^1 v^2 dx+2\int_0^1 v dx=0.
\end{split}
\right.
\end{equation}
By Lemma \ref{lm:norm ineq} and \eqref{eq:L^{infty} b'' 2},
$\|v\|_{C^2}=O(\sqrt{\beta})$.
Now we prove that for each $k\in\N$,
\begin{equation}\label{eq:induction beta->0}
\big\|\widetilde{b}\big\|_{C^2}=O(\sqrt{\beta}^k)
\Rightarrow
\big\|\widetilde{b}\big\|_{C^2}=O(\sqrt{\beta}^{k+1}).
\end{equation}
Then the statement is proved by induction. Assume
$\big\|\widetilde{b}\big\|_{C^2}=O(\sqrt{\beta}^k)$.
By $k=O(\sqrt{\beta})$ and \eqref{eq:v 2},
\[
\|v''\|_{L^{\infty}}\leq2(4|k|+3\sqrt{\beta})\|v\|_{L^{\infty}}
+3\sqrt{\beta}\|v^2\|_{L^{\infty}}=O(\sqrt{\beta}^{k+1}).
\]
By Lemma \ref{lm:norm ineq},
\[
\|v\|_{C^2}=O(\sqrt{\beta}^{k+1})
\]
and \eqref{eq:induction beta->0} is proved. This completes the proof.
\end{proof}
\subsection{The case $L\rightarrow\infty$}
To prove Theorem \ref{th:L->infty}, we use the following theorem which
is proved in Hamel, Fayard and Roques \cite{HFR}:
\begin{theorem}[Theorem 2.1 of \cite{HFR}]
Fix $\mu^+>0$ and $\mu^-\leq\mu^+$. Let
\[
\begin{split}
&\mu_L(x):=\sum_{k}(\mu^+\chi_{I_k}(x/L)+\mu^-\chi_{J_k}(x/L)),\\
&I_k:=k+[0,\theta],\ \ J_k:=k+(\theta,1).
\end{split}
\]
Then
\begin{equation}\label{eq:L->infty}
\lim_{L\rightarrow\infty}\,c^*(\mu_L)
=\min_{\lambda\geq(1-\theta)\sqrt{\mu^+-\mu^-}}\frac{j^{-1}(\lambda)}{\lambda},
\end{equation}
where the function $j:[\mu^+,\infty)\rightarrow[(1-\theta)\sqrt{\mu^+-\mu^-},\infty)$
is defined by
\[
j(m):=\theta\sqrt{m-\mu^+}+(1-\theta)\sqrt{m-\mu^-}.
\]
\end{theorem}
\begin{proof}[Proof of Theorem \ref{th:L->infty}]
Fix $\beta>0$ and $p\in(1,\infty)$. Let $\mu_{\theta}:=\beta^{1/p}\theta^{-/p}$ and
\[
b_{\theta,L}(x):=b_{\theta}(x/L),\ \
b_{\theta}(x):=\mu_{\theta}\sum_{k\in\Z}\chi_{I_k}(x)\ \ (I_k:=k+[0,\theta]).
\]
Then by \eqref{eq:L->infty},
\[
c^*_{\infty}(\theta):=\lim_{L\rightarrow\infty}\,c^*(b_{\theta,L})
=\min_{\lambda\geq(1-\theta)\sqrt{\mu_{\theta}}}\frac{j^{-1}(\lambda)}{\lambda}
=\frac{m_{\theta}}{j(m_{\theta})},
\]
where
\[
\begin{split}
&j(m)=\theta\sqrt{m-\mu_{\theta}}+(1-\theta)\sqrt{m},\\
&m_{\theta}:=\frac{8\mu_{\theta}\theta^2}
{3\theta^2+2\theta-1+(1-\theta)\sqrt{9\theta^2-2\theta+1}}.
\end{split}
\]
Thus the function $c^*_{\infty}(\theta)$ is continuous in $(0,1]$.
Moreover from
\[
\begin{split}
m_{\theta}&=\frac{8\mu_{\theta}\theta^2}{8\theta^2+O(\theta^4)}
=(1+O(\theta^2))\mu_{\theta},\\
j(m_{\theta})&=\theta\sqrt{m_{\theta}-\mu_{\theta}}+(1-\theta)\sqrt{m_{\theta}}
=(1+O(\theta))\sqrt{\mu_{\theta}}
\end{split}
\ \ {\rm as}\ \ \theta\rightarrow0,
\]
\[
c^*_{\infty}(\theta)=\frac{(1+O(\theta^2))\mu_{\theta}}{(1+O(\theta))\sqrt{\mu_{\theta}}}
=\frac{(1+O(\theta^2))}{(1+O(\theta))}\sqrt{\mu_{\theta}}
\rightarrow\infty
\ \ {\rm as}\ \ \theta\rightarrow0.
\]
Therefore for any $M>0$ and any $\ep\in(0,1)$, there is an $L_0>0$ such that
for any $L\geq L_0$,
\begin{equation}\label{eq:L->infty 1}
\max_{\theta\in[\ep,1]}c^*(b_{\theta,L})<\max_{\theta\in(0,1]}c^*(b_{\theta,L}),\ \
M<\max_{\theta\in(0,1]}c^*(b_{\theta,L}).
\end{equation}
In fact, by the monotonicity of $c^*(b_{\theta,L})$ in $L.>0$,
\begin{equation}\label{eq:c^* leq c^*_{infty}}
c^*(b_{\theta,L})\leq c^*_{\infty}(\theta)\ \ (\theta\in(0,1],\ L>0).
\end{equation}
By $\underset{\theta\rightarrow0}{\lim}\, c^*_{\infty}(\theta)=\infty$ and
$\underset{L\rightarrow\infty}{\lim}\,c^*(b_{\theta,L})=c^*_{\infty}(\theta)$ $(\theta\in(0,1])$,
for any $M.>0$ and $\ep\in(0,1)$, there are $\ep_0\in(0,\ep)$ and $L_0>0$ such that
for any $L\geq L_0$,
\[
\begin{split}
\max_{\theta\in[\ep,1]}c^*(b_{\theta,L})&\leq
\max_{\theta\in[\ep,1]}c^*_{\infty}(\theta)<c^*(b_{\ep_0},L)\leq c^*_{\infty}(\ep_0),\\
M&<c^*(b_{\ep_0,L})\leq\max_{\theta\in(0,1]}c^*(b_{\theta,L}).
\end{split}
\]
Thus \eqref{eq:c^* leq c^*_{infty}} holds.
Hence if we take $\theta(L)\in(0,1]$ such that
$c^*(b_{\theta(L),L})=\underset{\theta\in(0,1]}{\max}\,c^*(b_{\theta,L})$,
then by \eqref{eq:L->infty 1},
for any $M>0$ and any $\ep\in(0,1)$,
\[
0\leq\lim_{L\rightarrow\infty}\theta(L)<\ep,\ \ M\leq\lim_{L\rightarrow\infty}c^*(b_{\theta,L}).
\]
Thus
$\lim_{L\rightarrow\infty}\theta(L)=0,\ \
\lim_{L\rightarrow\infty}c^*(b_{\theta,L})=\infty$.
This completes the proof.
\end{proof}
\subsection{The case $\beta\rightarrow\infty$}
\begin{proof}[Proof of Theorem \ref{th:beta->infty}]
Let
\[
b_{\theta,L}(x):=\sum_{k\in\Z}\theta^{-1/p}\chi_{I_k}(x/L)\ \ (I_k:=k+[0,\theta]).
\]
Then by \eqref{eq:rescaling},
\[
\begin{split}
&k(\lambda,\beta^{1/p} b_{\theta,L})
=\beta^{1/p}k\Big(\beta^{-\frac{1}{2p}}\lambda,b_{\theta,\beta^{\frac{1}{2p}}L}\Big).\\
\therefore\ c^*(\beta^{1/p}b_{\theta,L})&=
\min_{\lambda>0}\frac{-k(\lambda,\beta^{1/p}b_{\theta,L})}{\lambda}
=\beta^{\frac{1}{2p}}\min_{\lambda>0}
\frac{-k\Big(\beta^{-\frac{1}{2p}}\lambda,b_{\theta,\beta^{\frac{1}{2p}}L}\Big)}
{\beta^{-\frac{1}{2p}}\lambda}\\&
=\beta^{\frac{1}{2p}}c^*\Big(b_{\theta,\beta^{\frac{1}{2p}}L}\Big).
\end{split}
\]
Therefore if we take $\theta(\beta)\in(0,1]$ such that
\[
c^*(\beta^{1/p}b_{\theta(\beta),L})=\max_{\theta\in(0,1]}c^*(b_{\theta,L}),
\]
then
\[
c^*\Big(b_{\theta(\beta),\beta^{\frac{1}{2p}}L}\Big)=
\max_{\theta\in(0,1]} c^*\Big(b_{\theta,\beta^{\frac{1}{2p}}L}\Big)
\]
and by Theorem \ref{th:L->infty},
\[
\theta(\beta)\rightarrow0,\ \
c^*(\beta^{1/p}b_{\theta(\beta),L})/\beta^{\frac{1}{2p}}
=\max_{\theta\in(0,1]} c^*\Big(b_{\theta,\beta^{\frac{1}{2p}}L}\Big)\rightarrow\infty
\ \ {\rm as}\ \ \beta\rightarrow\infty.
\]
This completes the proof.
\end{proof}
\appendix
\section{G$\hat{\rm a}$teaux derivative of $k(\lambda,b)$}
In this appendix, we calculate the derivative of the principal eigenvalue.

\vskip 6pt

\noindent
{\bf Lemma \ref{lm:smoothness of k}} (G$\hat{\rm a}$teaux derivative of $k(\lambda,b)$){\bf .}
{\it Fix} $p\in[1,\infty)$. {\it The maps}
\[
\begin{split}
&\R\times L^p_{L\mbox{\rm\tiny-per}}\ni(\lambda,b)\mapsto k(\lambda,b)\in\R,\\
&\R\times L^p_{L\mbox{\rm\tiny-per}}\ni
(\lambda,b)\mapsto \psi_{\lambda,b},\widetilde{\psi}_{\lambda,b}\in W^{2,p}_{L\mbox{\rm\tiny-per}}
\end{split}
\]
{\it are analytic and}
\begin{equation}\tag{\ref{eq:derivative of k}}
\begin{split}
\partial_b k(\lambda,b)[v]
&:=\lim_{h\rightarrow0}\frac{k(\lambda,b+h v)-k(\lambda,b)}{h}\\
&=-\langle\psi_{\lambda,b}\widetilde{\psi}_{\lambda,b},v\rangle
=-\int_0^L\psi_{\lambda,b}\widetilde{\psi}_{\lambda,b}v d x.
\end{split}
\end{equation}
\begin{proof}[Proof of Lemma \ref{lm:smoothness of k}]
Fix
$\lambda_{0}\in{\mathbb R}$,
$b_{0}\in L^{p}_{L\mbox{\tiny-per}}$
arbitrarily, and put
$k_{0}=k(\lambda_{0},b_{0})$,
$\psi_{0}=\psi_{\lambda_{0},b_{0}}$
and
$\widetilde{\psi}_{0}=\widetilde{\psi}_{\lambda_{0},b_{0}}$.
Let
$\epsilon>0$
be a constant such that
$(B_{\epsilon}(k_{0})\backslash\{k_{0}\})
\cap\sigma(-L_{\lambda_{0},b_{0}})=\emptyset$
and
\[
R_{\lambda,b}(\mu):=(\mu I+L_{\lambda,b})^{-1}
\ \mbox{ for }\mu\in\rho(-L_{\lambda,b})
={\mathbb C}\backslash\sigma(-L_{\lambda,b}),
\]
where
$B_{\epsilon}(k_{0})
:=\{\mu\in{\mathbb C}\mid |\mu-k_{0}|\leq\epsilon\}$
and
$\sigma(-L_{\lambda_{0},b_{0}})
:=\{\mu\in{\mathbb C}\mid
\mu I+L_{\lambda_{0},b_{0}}:
W^{2,p}_{L\mbox{\tiny-per}}\otimes{\mathbb C}
\rightarrow
L^{p}_{L\mbox{\tiny-per}}\otimes{\mathbb C}\
\mbox{is not invertible}\}$.

Then
\[
\begin{split}
&\psi_{0}\otimes\widetilde{\psi}_{0}
=\frac{1}{2\pi i}\int_{\partial B_{\epsilon}(k_{0})}
R_{\lambda_{0},b_{0}}(\mu)d\mu,\\
&k_{0}\psi_{0}\otimes\widetilde{\psi}_{0}
=\frac{1}{2\pi i}\int_{\partial B_{\epsilon}(k_{0})}\mu
R_{\lambda_{0},b_{0}}(\mu)d\mu,
\end{split}
\]
where the map
$\psi_{0}\otimes\widetilde{\psi}_{0}:
L^{p}_{L\mbox{\tiny-per}}\rightarrow L^{p}_{L\mbox{\tiny-per}}$
is defined by
$\phi\mapsto\langle\widetilde{\psi}_{0},\phi\rangle\,\psi_{0}$.

Moreover if we put
$R_{\lambda_{0},b_{0}}=R_{\lambda_{0},b_{0}}(\mu)$,
$R_{\lambda,b}=R_{\lambda,b}(\mu)$,
then
\begin{equation}\label{eq:Taylor}
\begin{split}
R_{\lambda,b}
=&R_{\lambda_{0},b_{0}}
\Big\{I-\Big((b_{0}-b+\lambda_{0}^{2}-\lambda^{2})I
+2(\lambda_{0}-\lambda)\frac{d}{d x}\Big)R_{\lambda_{0},b_{0}}
\Big\}^{-1}\\
=&R_{\lambda_{0},b_{0}}
\sum_{n=0}^{\infty}
\Big\{\Big((b_{0}-b+\lambda_{0}^{2}-\lambda^{2})I
+2(\lambda_{0}-\lambda)\frac{d}{d x}\Big)R_{\lambda_{0},b_{0}}
\Big\}^{n}
\end{split}
\end{equation}
in
$U_{\lambda_{0},b_{0}}
:=\{(\lambda,b)\in{\mathbb R}\times
L^{p}_{L\mbox{\tiny-per}}\mid
\|b-b_{0}\|+2|\lambda-\lambda_{0}|
+|\lambda^{2}-\lambda_{0}^{2}|<r_{0}\}$,
where
\[
r_{0}:=1/
\underset{\mu\in\partial B_{\epsilon}(k_{0})}{\max}
\|R_{\lambda_{0},b_{0}}(\mu)\|_{\mathcal{L}(L^{p}_{L\mbox{\tiny-per}},
W^{2,p}_{L\mbox{\tiny-per}})},
\]
\[
\begin{split}
&\mathcal{L}(X,Y)
:=\{L:X\rightarrow Y\mid L\mbox{ is a bounded linear operator}\},\\
&\|L\|_{\mathcal{L}(X,Y)}
:=\sup_{\phi\in X\backslash\{0\}}\|L\phi\|_{Y}/\|\phi\|_{X}
\end{split}
\]
for vector spaces
$X$, $Y$
with norms
$\|\cdot\|_{X}$, $\|\cdot\|_{Y}$,
respectively.

Hence
\[
\begin{split}
&\psi_{\lambda,b}\otimes\widetilde{\psi}_{\lambda,b}
=\frac{1}{2\pi i}\int_{\partial B_{\epsilon}(k_{0})}
R_{\lambda,b}(\mu)d\mu
\in\mathcal{L}(L^{p}_{L\mbox{\tiny-per}},
W^{2,p}_{L\mbox{\tiny-per}}),\\
&k(\lambda,b)
\psi_{\lambda,b}\otimes\widetilde{\psi}_{\lambda,b}
=\frac{1}{2\pi i}\int_{\partial B_{\epsilon}(k_{0})}\mu
R_{\lambda,b}(\mu)d\mu
\in\mathcal{L}(L^{p}_{L\mbox{\tiny-per}},
W^{2,p}_{L\mbox{\tiny-per}})
\end{split}
\]
are analytic in
$U_{\lambda_{0},b_{0}}$.

Thus the map
\[
U_{\lambda_{0},b_{0}}\ni(\lambda,b)\mapsto
\langle\widetilde{\psi}_{\lambda,b},\psi_{0}\rangle\psi_{\lambda,b}
\in W^{2,p}_{L\mbox{\tiny-per}}
\]
is analytic. Hence by $\langle\widetilde{\psi}_{\lambda,b},\psi_{0}\rangle\psi_{\lambda,b}>0$
for $(\lambda,b)\in U_{\lambda_{0},b_{0}}$, the map
\[
U_{\lambda_{0},b_{0}}\ni(\lambda,b)\mapsto\frac{\langle\widetilde{\psi}_{\lambda,b},\psi_{0}\rangle}
{\|\langle\widetilde{\psi}_{\lambda,b},\psi_{0}\rangle\psi_{\lambda,b}\|_2}\psi_{\lambda,b}
=\psi_{\lambda,b}
\in W^{2,p}_{L\mbox{\tiny-per}}
\]
is analytic.
Therefore the maps
$U_{\lambda_{0},b_{0}}\ni(\lambda,b)\mapsto
\widetilde{\psi}_{\lambda,b}\in W^{2,p}_{L\mbox{\tiny-per}}$
and
$U_{\lambda_{0},b_{0}}\ni(\lambda,b)\mapsto
k(\lambda,b)\in{\mathbb R}$
are analytic.

Put
$\partial\psi_{0}
:=\frac{d}{d \eta}\psi_{\lambda_{0},b_{0}+\eta v}\big|_{\eta=0}$
and
$\partial\widetilde{\psi}_{0}
:=\frac{d}{d \eta}\widetilde{\psi}_{\lambda_{0},b_{0}+\eta v}
\big|_{\eta=0}$,
then by
$\langle\psi_{\lambda,b},\widetilde{\psi}_{\lambda,b}\rangle=1$,
\[
0=\frac{d}{d \eta}\langle\psi_{\lambda_{0},b_{0}+\eta v},
\widetilde{\psi}_{\lambda_{0},b_{0}+\eta v}\rangle\Big|_{\eta=0}
=\langle \partial\psi_{0},\widetilde{\psi}_{0}\rangle
+\langle\psi_{0},\partial\widetilde{\psi}_{0}\rangle.
\]
By \eqref{eq:Taylor},
\[
R_{\lambda_{0},b_{0}+\eta v}-R_{\lambda_{0},b_{0}}
=-\eta R_{\lambda_{0},b_{0}}(v I)R_{\lambda_{0},b_{0}}
+O(\eta^{2}).
\]
Hence
\[
\begin{split}
\partial_{b}k(\lambda_{0},b_{0})[v]
=&\partial_{b}k(\lambda_{0},b_{0})[v]
+k_{0}\langle\widetilde{\psi}_{0},\partial\psi_{0}\rangle
+k_{0}\langle \partial\widetilde{\psi}_{0},\psi_{0}\rangle\\
=&\frac{d}{d \eta}\big(k(\lambda_{0},b_{0}+\eta v)
\langle\widetilde{\psi}_{0},
\psi_{\lambda_{0},b_{0}+\eta v}
\rangle
\langle\widetilde{\psi}_{\lambda_{0},b_{0}+\eta v},
\psi_{0}\rangle\big)\Big|_{\eta=0}\\
=&
\Big\langle\widetilde{\psi}_{0},
\frac{-1}{2\pi i}\int_{\partial B_{\epsilon}(k_{0})}\mu
R_{\lambda_{0},b_{0}}(\mu)(v I)R_{\lambda_{0},b_{0}}(\mu)\psi_{0}d\mu\Big\rangle\\
=&-\langle\psi_{0}\widetilde{\psi}_{0},v\rangle.
\end{split}
\]
\end{proof}
\section{Projection onto the constraint manifold}
In this appendix, we prove the following lemma.

\vskip 6pt

\noindent
{\bf Lemma \ref{lm:projection}} (Projection operator){\bf.}
{\it Put} $b_0:=f^{-1}(\beta)$, $\ep_0:=\|b_0\|_{\infty}=f^{-1}(\beta)$ {\it and}
\[
B_{\ep_0}:=\Big\{b_0+v\,\Big|\,
v\in L^{\infty}_{L\mbox{\rm\tiny-per}},\ \int_0^L v dx=0,\ \|v\|_{\infty}<\ep_0
\Big\}.
\]
{\it Let} $P$ {\it be a map defined as follows}:
\[
B_{\ep_0}\ni b\mapsto P(b):=\mu(b)b\in A_f,
\]
{\it where} $\mu=\mu(b)$ {\it is a real number with}
$\int_0^L f(\mu b)dx=\beta L$. {\it Then}
\begin{equation}\tag{\ref{eq:projection}}
\begin{split}
P(b_0+v)=&b_0+v-\frac{f''(b_0)}{2f'(b_0)L}\|v\|_2^2+o(\|v\|_2^2)\\
&{\rm as}\ \ B_{\ep_0}\ni b_0+v\rightarrow b_0\ \ {\rm in}\ \ L^{\infty}.
\end{split}
\end{equation}
\begin{lemma}\label{lm:mu}
Let $b_0:=f^{-1}(\beta)$ and
\[
B:=\Big\{b_0+v\,\Big|\,v\in L^{\infty}_{L\mbox{\rm\tiny-per}},\
\int_0^L v dx=0,\ \|v\|_{\infty}< b_0\Big\}
\]
Let $\mu=\mu(b)$ be a real number with $\int_0^L f(\mu b)dx=\beta L$
for $b=b_0+v\in B$. Then for $b=b_0+v\in B$,
\begin{equation}\label{eq:mu}
\mu(b)=1-\frac{f''(b_0)}{2f'(b_0)b_0 L}\|v\|_2^2+o(\|v\|_2^2)
\ \ {\rm as}\ \ b\rightarrow b_0\ \ {\rm in}\ \ L^{\infty}.
\end{equation}
\end{lemma}
\begin{proof}[Proof of Lemma \ref{lm:projection}]
Let $P$ be the map defined as in Lemma \ref{lm:projection}. Then by Lemma \ref{lm:mu},
for $b=b_0+v\in B$,
\[
\begin{split}
P(b)&=\mu(b)b=\Big\{1-\frac{f''(b_0)}{2f'(b_0)b_0 L}\|v\|_2^2+o(\|v\|_2^2)\Big\}(b_0+v)\\
&=b_0+v-\frac{f''(b_0)}{2f'(b_0)L}\|v\|_2^2+o(\|v\|_2^2).
\end{split}
\]
This completes the proof.
\end{proof}
\begin{lemma}\label{lm:small order}
Let $g$ be a continuous function with $g(0)=0$. Then
\[
\lim_{\|v\|_{\infty}\rightarrow0}
\frac{\int_0^L g(v)v^2 dx}{\|v\|_2^2}=0.
\]
\end{lemma}
\begin{proof}
By continuity of $g$,
\[
\begin{split}
\Big|\frac{\int_0^L g(v)v^2 dx}{\|v\|_2^2}\Big|
&\leq\frac{\int_0^L |g(v)|v^2 dx}{\|v\|_2^2}\\
&\leq \sup_{|t|\leq\|v\|_{\infty}}|g(t)|\rightarrow0\ \
{\rm as}\ \ \|v\|_{\infty}\rightarrow0.
\end{split}
\]
\end{proof}
\begin{proof}[Proof of Lemma \ref{lm:mu}]
Fix $\ep_1\in(0,b_0)$ and put
\[
\begin{split}
m&:=\underset{\tiny
\begin{array}{c}
\|v\|_{\infty}\leq\ep_1,\\
\int_0^L v dx=0
\end{array}}{\inf}\, \mu(b_0+v)\,\inf_{x}\,(b_0+v(x)),\\
M&:=\underset{\tiny
\begin{array}{c}
\|v\|_{\infty}\leq\ep_1,\\
\int_0^L v dx=0
\end{array}}{\sup}\, \mu(b_0+v)\,\sup_{x}\,(b_0+v(x)).
\end{split}
\]
Fix $v\in L^{\infty}_{L\mbox{\tiny-per}}$
with $\int_0^L v dx=0$, $\|v\|_{\infty}\leq \ep_1$.
Put $b_{\ep}:=b_0+\ep v$ and $\mu=\mu(\ep):=\mu(b_{\ep})$.
Then by differentiating both sides of $\int_0^L f(\mu b_{\ep})=\beta L$ with respect to $\ep$,
\begin{equation}\label{eq:mu' 1}
\mu'\int_0^L f'(\mu b_\ep) dx+\mu\int_0^L f'(\mu b_{\ep})b_{\ep}dx=0.
\end{equation}
By $\int_0^L v dx=0$,
\begin{equation}\label{eq:1}
\begin{split}
\Big|\int_0^L f'(\mu b_{\ep}) v dx\Big|
&=\Big|\int_0^L\{f'(\mu b_{\ep})-f'(\mu b_0)\}v dx\Big|\\
&\leq \sup_{m\leq t\leq M}|f''(t)|\,\mu\ep\|v\|_2^2=O(\ep\|v\|_2^2).
\end{split}
\end{equation}
Hence by \eqref{eq:mu' 1},
\begin{equation}\label{eq:mu' 2}
\mu'(\ep)=\frac{-\mu\int_0^L f'(\mu b_{\ep}) v dx}{\int_0^L f'(\mu b_{\ep})b_{\ep}dx}
=O(\|v\|_2^2).
\end{equation}
By $\int_0^L v dx=0$, $\mu(b_0)=1$ and \eqref{eq:mu' 1},
\[
\begin{split}
\Big|&\int_0^L f'(\mu b_{\ep}) b_{\ep} dx-\int_0^L f'(b_0)b_0 dx\Big|\\
&\leq \Big|\int_0^L \{f'(\mu b_{\ep}) -f'(\mu b_0)\}b_{\ep} dx\Big|
+\ep\Big|\int_0^L \{f'(\mu b_0) -f'(b_0)\} b_0 dx\Big|
=O(\|v\|_2).
\end{split}
\]
\begin{equation}\label{eq:2}
\therefore\ \int_0^L f'(\mu b_{\ep}) b_{\ep} dx=\int_0^L f'(b_0)b_0 dx+O(\|v\|_2).
\end{equation}
Differentiating both sides of \eqref{eq:mu' 1},
\[
\mu''\int_0^L f'(\mu b_{\ep})b_{\ep} dx+\mu^2\int_0^L f''(\mu b_{\ep})v^2 dx+\mu'R(v)=0,
\]
where
\[
\begin{split}
R(v)&:=
2\int_0^L f'(\mu b_{\ep})v dx+\mu'\int_0^L f''(\mu b_{\ep})b_{\ep}^2 dx
+\mu\int_0^L f''(\mu b_{\ep})b_{\ep}v dx\\
&=O(\|v\|_2).
\end{split}
\]
Hence by \eqref{eq:mu' 2},
\begin{equation}\label{eq:mu'' 1}
\mu''\int_0^L f'(\mu b_{\ep})b_{\ep} dx+\mu^2\int_0^L f''(\mu b_{\ep})v^2 dx+O(\|v\|_2^3)=0.
\end{equation}
On the other hand, it holds that
\begin{equation}\label{eq:3}
\mu^2\int_0^Lf''(\mu b_{\ep})v^2 dx=\int_0^L f''(b_0)v^2 dx+o(\|v\|_2^2).
\end{equation}
In fact,
\[
\mu^2\int_0^L f''(\mu b_{\ep})v^2 dx=\int_0^L f''(b_0)v^2dx+{\rm I}+{\rm II}+{\rm III},
\]
where
\[
\begin{split}
{\rm I}:=&(\mu-1)\int_0^L f''(\mu b_{\ep})v^2 dx,\ \
{\rm II}:=\int_0^L \{f''(\mu b_{\ep})-f''(b_{\ep})\}v^2 dx,\\
{\rm III}:=&\int_0^L\{f(b_{\ep})-f(b_0)\}v^2dx.
\end{split}
\]
From Lemma \ref{lm:small order},
\[
{\rm III}=o(\|v\|_2^2).
\]
By $\mu(b_0)=1$ and \eqref{eq:mu' 2},
\[
{\rm I}=O(\|v\|_2^4),\ \
|{\rm II}|\leq\sup_{\tiny
\begin{array}{c}
m\leq s\leq M,\\
|t|\leq O(\|v\|_2^2)
\end{array}}
|f''(s+t)-f''(s)|\,\|v\|_2^2=o(\|v\|_2^2).
\]
Thus \eqref{eq:3} holds. Hence by \eqref{eq:2}, \eqref{eq:mu'' 1} and \eqref{eq:3},
\[
\begin{split}
\mu''(\ep)&
=\frac{-\mu^2\int_0^L f''(\mu b_{\ep})v^2dx+O(\|v\|_2^3)}
{\int_0^L f'(\mu b_{\ep})b_{\ep} dx}
=\frac{-\int_0^L f''(b_0)v^2 dx+o(\|v\|_2^2)}{\int_0^L f'(b_0)b_0 dx+O(\|v\|_2)}\\
&=\frac{-\int_0^L f''(b_0)v^2 dx}{\int_0^L f'(b_0)b_0 dx}+o(\|v\|_2^2)
=\frac{-f''(b_0)}{f'(b_0)b_0 L}\|v\|_2^2+o(\|v\|_2^2).
\end{split}
\]
Therefore
\[
\begin{split}
\mu(b_0+v)&=\mu(1)=\mu(0)+\mu'(0)+\int_0^L (1-\ep)\mu''(\ep)d\ep\\
&=1-\frac{f''(b_0)}{2f'(b_0)b_0 L}\|v\|_2^2+o(\|v\|_2^2).
\end{split}
\]
\end{proof}
\begin{center}
\acknowledgments
\end{center}

The authors would like to thank Prof. Matano for many helpful suggestions and
continuous encouragement.

\bibliographystyle{amsalpha}

\end{document}